\documentclass{amsart}

\usepackage{amsmath,amssymb}
  \usepackage{paralist}
  \usepackage{graphics} 
  \usepackage{epsfig} 
\usepackage{amsthm}
\usepackage{tikz}
\usepackage[all]{xy}
 \usepackage[colorlinks=true]{hyperref}
 \usepackage{verbatim}

\hypersetup{urlcolor=blue, citecolor=red}
\usepackage{hyperref}

\newtheorem{theorem}{Theorem}[section]
\newtheorem{proposition}[theorem]{Proposition}
\newtheorem{corollary}[theorem]{Corollary}
\newtheorem{lemma}[theorem]{Lemma}
\newtheorem{main lemma}[theorem]{Main Lemma}

{\theoremstyle{definition}
\newtheorem{definition}{Definition}

\newtheorem{remark}{Remark}

\newcommand{\R}{\mathbb{R}}
\newcommand{\Z}{\mathbb{Z}}
\newcommand{\N}{\mathbb{N}}

\def\beq{\begin{equation}}
\def\eeq{\end{equation}}

\def\pa{\partial}
\def\t{\theta}

\def\d{\delta}

\def\wt{\widetilde}
\def\wh{\widehat}
\def\f{\varphi}
\def\l{\lambda}
\def\a{\alpha}

\def\n{\nabla}
\def\o{\omega}
\def\eps{\varepsilon}
\def\r{\rho}
\def\wc{\rightharpoonup}
\newcommand{\ddfrac}[2] {\frac{\displaystyle #1 }{\displaystyle #2} }

\DeclareMathOperator{\dist}{dist}
\DeclareMathOperator{\ind}{Ind}
\DeclareMathOperator{\loc}{loc}

\title[Avoiding collisions with topological constraints]{Avoiding collisions under topological constraints in variational problems coming from celestial mechanics}
\author{Nicola Soave and Susanna Terracini}

\address{
\hbox{\parbox{5.7in}{\medskip\noindent
 Nicola Soave\\
Mathematisches Institut, \\
Justus-Liebig-Universit\"at Gie{\ss}en,\\
Arndtstra{\ss}e 2, 35392 Gie{\ss}en (Germany)\\[2pt]
{\em{E-mail address: }}{\tt nicola.soave@gmail.com.}\\[5pt]
 S. Terracini\\
 Dipartimento di Matematica ``Giuseppe Peano'',\\
Universit\`a di Torino, \\
Via Carlo Alberto 10,
10123 Torino (Italy). \\[2pt]
                                     \em{E-mail address: }{\tt susanna.terracini@unito.it.}}}
}

\date{\today}

\thanks{
\indent{2010} {\it Mathematics Subject Classification:} 70F16, 49A10, 70F05,  (70F10, 70K05)\\
\indent The authors are partially supported through the project ERC Advanced Grant 2013 n. 339958 ``Complex Patterns for Strongly Interacting Dynamical Systems - COMPAT". 
 }

\begin{document}
\maketitle



%

\begin{flushright}
\emph{To Yvonne Choquet-Bruhat with great admiration}
\end{flushright}
\begin{abstract}
In a singular potential setting, we generalize a method which allows to show that minimizers under topological constraints of the action functional (or of the Maupertuis') are collision-free. This methods applies to $3$-dimensional problems of celestial mechanics exhibiting a particular cylindrical symmetry, as well as to planar problems of $N$-centre type, where it gives optimal results.
\end{abstract}

\section{Introduction and main results}\label{sec: intro}

In the $N$-body problem, a \emph{collision trajectory} $q(t)=(q_1(t),\dots,q_N(t))$ has at least one instant $\bar t$ such that $q_i(\bar t)=q_j(\bar t)$ for some $i \neq j$. As collisions represent the main obstruction in order to apply minimization arguments, in the last decades several arguments have been developed to insure that an action minimizing path is free of collisions, starting from \cite{SeTe,SeTe2,Ta1,Ta2} . One of the most powerful tools has to be ascribed to Marchal, see \cite{Chencin,Mar}, and can be summarized in the following statement: let $Q_1, Q_2$ two different configurations in the $N$-body problem, and let $H(Q_1,Q_2,T)$ be the space of $H^1$-functions connecting $Q_1$ and $Q_2$ in a given time $T$: then,  any action minimizer in $H(Q_1,Q_2,T)$ is collision-free. It is instructive to recall the sketch of the proof. Let $\bar q$ be a minimizing collision trajectory; let us construct a family of rigid variations, parametrized over the sphere, moving one of the colliding particles away from the collision. While  it may be hard, or even impossible,  to evaluate the action variation for a single element of the family, it is fairly easy to estimate the average of the perturbed action over all the perturbations, and show that this average is smaller than the action of $\bar q$, concluding that $\bar q$ cannot be a minimizer. This idea has been widely generalized in \cite{Fe, FeTe} in order to include symmetries, several potentials,  and a large numer of applications to the search of periodic solutions to the $N$-body problem. Very general results on the absence of collisions for minimizers of the Bolza problem are reported in \cite{BaFeTe08}.

On the other hand, when looking for new selected trajectories in Celestial Machanics, one often seeks minimizers of the action in some set of functions sharing a prescribed topological behaviour, as, for example, in \cite{Chen1,Chen2,FuGr,FuGrNe,knauf1992,Shy,TeVe,Vent}. In such a situations, Marchal's lemma cannot be employed because the average argument may destroy the topological constraint, and this, usually, makes impossible to deduce any conclusive information. As a typical example, one can think at the hip-hop trajectories constructed in \cite{TeVe}, where the authors introduced a method, adapted also in different situations in \cite{Ca,SoTe}, to prove that minimizers of the action functional (or of the Maupertuis' one, which we define in what follows) are, under suitable topological constraints, collision-free.  The existence of collision free action minimizing periodic trajectories with nontrivial homotopy type has been recently linked to the threshold of existence of minimizing parabolic trajectories in \cite{BaTeVe}. Here we aim at giving further generalizations and unified approaches to the different problems already considered in the quoted literature. The method we are going to describe, which is topological in nature, provides optimal results in the planar case, even if it works also for some $3$-dimensional problems, under stronger assumptions. In order to avoid misunderstanding, we present separately the spatial problem and the planar one. In both the situations, we will deal with a configuration space which is not simply connected; we remark that we can deal with both the action functional and the Maupertuis' one without substantial differences. 

Our main results are stated in a form as general as possible, in order to obtain tools with a wide applicability. This fact led us to consider a setting which may appear quite far away from concrete problems. This is not the case, as we show presenting two applications at classical singular problems, in Section \ref{sec: appl}.


\subsection{Main result for the $3$-dimensional problem}

Let $V\in \mathcal{C}^1(\R^3 \setminus \Sigma)$, where $\Sigma$ is the union of a finite number of infinite straight lines $r_1,\dots,r_m$. We assume that $V \ge 0$ and $V(q) \to +\infty$ as $\dist(q,\Sigma) \to 0$.

We search for classical solution of the motion equation
\begin{equation}\label{intro, motion eq}
 \ddot{q}(t)= \nabla V(q(t)).
\end{equation}

For a fixed $k$, let $\Pi_k$ be a plane orthogonal to $r_k$, and let us consider the splitting $\R^3 = \Pi_k \oplus r_k$. Using, as usual, the complex notation for points of $\R^2$, we introduce a system of cylindrical coordinates $u^k= \rho^k \exp\{i \t^k\}$ in the plane $\Pi_k$, and $z^k$ on the line $r_k$, in such a way that 
\begin{itemize}
\item $q\in \{z^k=0\}$ if and only if $q \in \Pi_k$;
\item $q \in \{u^k=0\}$ if and only if $q \in r_k$;
\item chosen an orientation for the plane $\Pi_k$, the angle $\t^k$ is counted in counterclockwise sense.
\end{itemize}
We often omit the index $k$ when there is not possibility of misunderstanding, to simplify the notation. 

It is possible to define the concept of angle with respect to the axis $r_k$ for an ordered pair of points $p_1,p_2 \in \R^3 \setminus r_k$; firstly, up to a rotation we can assume that the angular component of $p_1$ is $0$. In these coordinates,
\[
p_1 = (\rho_1 \exp\{i 0\},z_1) \quad \text{and} \quad p_2 = (\rho_2 \exp\{i \theta_2\},z_2),
\]
where it is not restrictive to choose $\theta_2 \in [0,2\pi)$.
\begin{definition}\label{def: angle}
We say that $\t_2$ is \emph{the angle between $p_
1$ and $p_2$ with respect to the axis $r_k$}. 
\end{definition}

Let us consider, for any $k=1,\ldots,m$, a neighbourhood $\Xi_k$ of the line $r_k$ of type
\[
\Xi_k=\{ p \in \R^3: \dist(p,r_k)<d_k\}.
\]
Let $p_1,p_2 \in \partial \Xi_k$; we define
\[
\wh{K}= \wh{K}_{p_1 p_2}\left([a,b]\right) := \left\{ q \in H^1([a,b],\R^3) \left| \begin{array}{l}
q(a)=p_1,\  q(b)=p_2,  \\
q(t) \in \Xi_k \setminus r_k \text{ for every $t\in (a,b)$} \end{array}\right.\right\rbrace ,
\]
The set $\wh{K}$ has several connected components, which can be determined according to the winding number of their elements with respect to the axis $r_k$. To compute the rotation number of $q \in \wh{K}$ we consider the projection $u^k:[a,b] \to \overline{\Xi}_k \cap \Pi_k$ of $q$ on the plane $\Pi_k$; as $q \in \wh{K}$, it is well defined the usual winding number
\begin{equation}\label{usual rot number}
\ind(u^k,0):= \int_{u^k([a,b])} d \theta^k.
\end{equation}
Clearly, it results $\ind(u^k,0) = \hat \theta + 2 l \pi$ for some $l \in \Z$, where $\hat \theta$ is the angle between $p_1$ and $p_2$ with respect to the axis $r_k$.
\begin{definition}\label{def: winfding number}
We define the \emph{rotation number of $q \in \wh{K}_{p_1 p_2}([a,b])$ with respect to the axis $r_k$} as 
\[
\ind(q,r_k):=  \ind(u^k,0) \in \hat \theta + 2\pi\Z,
\]
where the right hand side has been defined by \eqref{usual rot number}.
\end{definition}
A connected component of $\wh{K}$ is of type
\[
\wh{K}_l=\wh{K}_l^{p_1 p_2}([a,b]) := \left\{ q \in \wh{K}: \ind(q,r_k)=\hat \theta +2l\pi\right\},
\]
where $l \in \Z$ is an arbitrary integer number. We denote with $K_l=K_l^{p_1 p_2}([a,b])$ the closure of $\wh{K}_l$ with respect to the weak topology of $H^1([a,b],\R^3)$; since the weak $H^1$ convergence implies the uniform one, $K_l \setminus \wh{K}_l$ consists in collision functions and functions parametrizing paths leaning on the boundary $\pa \Xi_k$.

\medskip

We now specify the shape of the potentials which we can deal with. It is convenient to introduce the following definition.

\begin{definition}
let $\gamma=q([c,d])$, where $q:[c,d] \to \R^3$ is continuous. Let us consider the expression $q=(u^k,z^k)$, for some $k =1,\dots, m$. Let
\[
\rho^\gamma:= \max \left\{|u^k(t)|: t \in [c,d] \right\},
\]
and let
\[
z^\gamma_{\min}:= \min \left\{z^k(t): t \in [c,d] \right\} \quad \text{and} \quad z^\gamma_{\max}:= \max \left\{z^k(t): t \in [c,d] \right\}.
\]
A \emph{cylindrical neighbourhood} of $\gamma$ is a set of type
\[
\left\{p=(u^k,z^k) \in \Xi_k: \text{ $|u^k|<\rho^\gamma +\delta$ and $z^k \in \left( z^\gamma_{\min}-\delta, z^\gamma_{\max}+\delta \right)$}\right\},
\]
for some $\delta>0$.
\end{definition}

\begin{definition}\label{def: local interaction}
Let $V: \R^3 \setminus \Sigma \to \R$, and let $q \in H^1([a,b],\R^3)$. We say that \emph{the interaction between $q$ and $\Sigma$ is locally axially Keplerian} if for any $k=1,\dots,m$ there exists a neighbourhood $\Xi_k= \{p \in \R^3: \dist (p,r_k)<d_k\}$ of $r_k$ such that for every connected component $\gamma \subset \overline{\Xi}_k \cap q([a,b])$ there exists a cylindrical bounded neighbourhood $\Gamma \subset \overline{\Xi}_k$ of $\gamma$ such that
\[
V(q)= \frac{m_k}{\alpha_k |u^k|^{\alpha_k}} + V_0(|u^k|, z^k) \qquad \forall q = (u^k,z^k) \in \Gamma,
\]
where $m_k>0$, $\a_k \in (0,2)$, and $V_0 \in \mathcal{C}^1(\overline{\Gamma})$.
\end{definition}

\begin{remark}\label{rem: su V 3D}
$i$) To understand the meaning of this definition, it is useful to think at the following situation. Let us write $\R^3 \ni q= (u,z) \in \R^2 \times \R$, and let $V$ be globally defined by 
\[
\bar V(q) = \frac{m}{\alpha |u|^\alpha} + \bar V_0 (|u|,z),
\] 
where $V_0 \in W^{1,\infty}(\R_+ \times \R) \cap \mathcal{C}^1(\R_+ \times \R)$ and $m>0$. Clearly, the interaction between $q$ and $\bar V$ is (locally) axially Keplerian for any $q \in H^1([a,b],\R^3)$. To assume that the interaction between $q \in H^1([a,b],\R^3)$ and a potential $V$ is locally axially Keplerian means that, at least locally when $q$ approaches the singular set $\Sigma$, the potential acts on $q$ as $\bar V$. \\
$ii$) Since $\bar V_0$ depends on $u$ only through $|u|$, if $q \in H^1([a,b],\R^3)$ is a solution of $\ddot{q}(t) = \nabla \bar V(q(t))$ in $(a,b)$, then the angular momentum of $q$ with respect to the axis $\{u=0\}$, which we define as $u \land \dot{u}$, is constant in $(a,b)$. 
$iii$) Let $h \in \R$. As $\lim_{\dist(q,\Sigma) \to 0} V(q)=+\infty$, it is possible, if necessary, to replace $d_k$ with a smaller quantity, in such a way that 
\[
\overline{\Xi}_k \subset \left\{q \in \R^3: V(q) +h \ge C > 0 \right\},
\]
for some $C>0$, for every $k$. The set on the left hand side, called \emph{the Hill's region}, is the set in which any solution to equation \eqref{intro, motion eq} having energy $h$ is confined. 
\end{remark}

We now recall the definition of the functionals we deal with. For $[a,b] \subset \R$ fixed, the \emph{action functional} $\mathcal{A}_{[a,b]}: H^1([a,b],\R^3) \to \R \cup \{+\infty\}$ is
\[
\mathcal{A}_{[a,b]}(q):= \int_a^b \left(\frac{1}{2} |\dot{q}(t)|^2 + V(q(t)) \right)\,dt.
\]
For $h \in \R$, the \emph{Maupertuis' functional} $\mathcal{M}_h([a,b];\cdot):H^1([a,b],\R^3)\to \R \cup \{+\infty\}$ is
\[
\mathcal{M}_h([a,b];u):= \frac{1}{2} \int_a^b |\dot{q}(t)|^2\,dt \int_a^b\left(V(q(t))+h \right)\,dt.
\]
We often write $\mathcal{M}_h$ instead of $\mathcal M_h\left([a,b];\cdot\right)$ when there is not be possibility of misunderstanding.

Let us consider a boundary value problem associated to equation \eqref{intro, motion eq}; typical examples are the periodic problem and the fixed ends one. It is well known that collision-free critical points of $\mathcal A_{[a,b]}$ in suitable subsets of $H^1([a,b],\R^3)$ (the choice of the subset will depend on the particular boundary value problem we are considering) are classical solution of \eqref{intro, motion eq} in the time interval $[a,b]$. Also, collision-free critical points of $\mathcal M_h$ at a positive level, suitably re-parametrized, are classical solution of \eqref{intro, motion eq} with energy $h$. We refer to the appendix at the end of the paper for more details. 

\medskip
 
We are ready to state the first of our main results.

Let $V: \R^3 \setminus \Sigma \to \R$, and let $\bar q \in H^1([a,b],\R^3)$. Assume that the interaction between $\bar q$ and $\Sigma$ is locally axially Keplerian, and let $\Xi_k$ the neighbourhood of $r_k$ given by Definition \ref{def: local interaction}. Let $\gamma= \bar q([c,d])$ be a connected component of $\overline{\Xi}_k \cap \bar q([a,b])$, for some $k = 1,\dots,m$.
Let $\hat \theta$ be the angle between $\bar q(c)$ and $\bar q(d)$ with respect to the axis $r_k$, as introduced in Definition \ref{def: angle}. Note that $\bar q|_{[c,d]} \in K_l$ for some $l \in \Z$. 

\begin{theorem}\label{thm: main 1}
In the previous setting, assume that $\bar q|_{[c,d]}$ is a minimizer of the action functional $\mathcal A_{[c,d]}$ or of the Maupertuis' functional $\mathcal M_h$ (for some $h \in \R$) in $K_l$. If $\hat \theta + 2l\pi \in \left(0,2\pi/(2-\alpha_k)\right)$, then $\bar q$ has no collision with $r_k$ in $[c,d]$. 
\end{theorem}

\subsection{Main results in the planar case}\label{sub: intro planar}

In the planar case, we assume that $V \in \mathcal{C}^1(\R^2 \setminus \Sigma)$, where $\Sigma$ is the union of a finite number of points $c_1,\dots,c_m$; we suppose that $V \ge 0$ and $V(u) \to +\infty$ as $\dist(u,\Sigma) \to 0$.

We keep here the same notation introduced above for the sets of functions and for the action and the Maupertuis' functionals. 

\medskip

Let us fix $k=1,\dots,N$. For $p_1,p_2 \in \R^2 \setminus \{c_k\}$,  we introduce a system of polar coordinates $u=c_k+\rho^k \exp\{i\t^k\}$ centred in the point $p_k$. Up to a rotation we can assume that the angular coordinate of $p_1$ is $0$, so that 
\[
p_1 = c_k+\rho_1^k \exp\{i 0\} \quad \text{and} \quad p_2 = c_k+\rho_2^k \exp\{i \theta_2^k\},
\]
where it is not restrictive to choose $\theta_2 \in [0,2\pi)$. We often omit the index $k$.
\begin{definition}\label{def: angle planar}
We say that $\theta_2$ is \emph{the angle between $p_1$ and $p_2$ with respect to the pole $c_k$}.
\end{definition}

Let $\Xi_k=\{\dist(u,c_k)<d_k\}$ (for some $d_k>0$). Let $p_1,p_2 \in \pa \Xi_k$. We define $\wh{K}_{p_1 p_2}([a,b])$ as the set of $H^1([a,b],\R^2)$ collision-free functions connecting $p_1$ and $p_2$.

\begin{definition}
We define the \emph{rotation number of $u \in \wh{K}_{p_1 p_2}([a,b])$ with respect to the pole $c_k$} as 
\[
\ind(u,c_k):=  \int_{u([a,b])} d\t^k \in \hat \theta + 2\pi\Z,
\]
where $\hat \theta$ is the angle between $p_1$ and $p_2$ with respect to the pole $c_k$.
\end{definition}
A connected component of $\wh{K}$ is of type
\[
\wh{K}_l=\wh{K}_l^{p_1 p_2}([a,b]) := \left\{ u \in \wh{K}: \ind(u,c_k)=\hat \theta +2l\pi\right\},
\]
where $l \in \Z$ is an arbitrary integer number. We denote with $K_l=K_l^{p_1 p_2}([a,b])$ the closure of $\wh{K}_l$ with respect to the weak topology of $H^1([a,b],\R^3)$.

\medskip

As far as the functional $V$, we replace the notion of being locally axially Keplerian with the following.

\begin{definition}\label{def: planar kep}
Let $V: \R^2 \setminus \Sigma \to \R$, and let $u \in H^1([a,b],\R^2)$. We say that \emph{the interaction between $u$ and $\Sigma$ is locally Keplerian} if for any $k=1,\dots,m$ there exists a neighbourhood $\Xi_k= \{\dist(u,c_k)<d_k\}$ of $c_k$ such that, for every connected component $\gamma \subset \overline{\Xi}_k \cap u([a,b])$, it results
\[
V(u)= \frac{m_k}{\alpha_k |u-c_k|^{\alpha_k}} + V_0(u) \qquad \forall u \in \Xi_k,
\]
where $m_k>0$, $\alpha_k \in (0,2)$, and $V_0 \in \mathcal{C}^1\left(\overline{\Xi}_k\right)$.
\end{definition}

\begin{remark}
We point out that, with respect to the $3$-dimensional case, we do not require that $V_0$ depends on $u$ only through its radial component $\rho^k$. 
\end{remark}

Firstly, we can recover the natural extension of Theorem \ref{thm: main 1} in the planar case. Let $\bar u \in H^1([a,b],\R^2)$. Assume that the interaction between $\bar u$ and $\Sigma$ is locally Keplerian, with the further assumption $V_0=V_0(|u-c_k|)$, and let $\Xi_k$ be the neighbourhood of $c_k$ given by Definition \ref{def: planar kep}. Let $\gamma= \bar u([c,d])$ be a connected component of $\overline{\Xi}_k \cap \bar u([a,b])$, for some $k = 1,\dots,m$. Let $\hat \theta$ be the angle between $\bar u (c)$ and $\bar u(d)$ with respect to the centre $c_k$, as introduced in Definition \ref{def: angle planar}. Note that $\bar u|_{[c,d]} \in K_l$ for some $l \in \Z$.

\begin{corollary}\label{thm: main 2}
In the previous setting, assume that $\bar u|_{[c,d]}$ is a local minimizer of $\mathcal A_{[c,d]}$ or of $\mathcal M_h$ (for some $h \in \R$) in $K_l$. If $\hat \theta + 2l\pi \in \left(0,2\pi/(2-\alpha_k)\right)$, then $\bar u$ has no collision in $c_k$ in $[c,d]$. 
\end{corollary}

In the planar case we can say something more on the possibility that a minimizer of the action functional or of the Maupertuis' one has a collision. Let us focus on the definition of $\wh{K}$ and of its connected components. We collect together the components of $\wh{K}$ with rotation number "having the same parity", defining
\begin{align*}
\wh{K}_{e}=\wh{K}_{e}^{p_1 p_2}([a,b]):= \bigcup_{l \in 2\Z} \wh{K}_l^{p_1 p_2}([a,b]) \qquad &(\text{direct class})\\
\wh{K}_{o}=\wh{K}_{o}^{p_1 p_2}([a,b]):= \bigcup_{l \in 2\Z+1} \wh{K}_l^{p_1 p_2}([a,b]) \qquad &(\text{inverse class}),
\end{align*}
and their closure $K_{e}$ and $K_{o}$ with respect to the weak topology of $H^1$.

\begin{definition}
A (local) \emph{collision-ejection minimizer} of $\mathcal A_{[a,b]}$ or of $\mathcal M_h$ is a (local) minimizer $u_{\min}$ of $\mathcal A_{[a,b]}$ or of $\mathcal M_h$ in $U \subset H^1([a,b],\R^2)$, such that:
\begin{itemize}
\item[($i$)] there exists a collision set $T_c(u_{\min})\subset [a,b]$, such that for every $t^* \in T_c(u_{\min})$ it holds $u_{\min}(t^*) \in \Sigma$;
\item[($ii$)] the function is symmetric with respect to each collision instant, in the sense that
\[
u_{\min}(t+t^*)=u_{\min}(t^*-t) \qquad \forall t^* \in T_c(q), 
\]
where the previous relation holds whenever $t \in \R$ is such that both $t+t^*$ and $t^*-t$ are in $[a,b]$.
\end{itemize}
\end{definition}

\medskip

Let $V: \R^2 \setminus \Sigma \to \R$, and let $\bar u \in H^1([a,b],\R^3)$. Assume that the interaction between $\bar u$ and $\Sigma$ is locally Keplerian, and let $\Xi_k$ be the neighbourhood of $c_k$ given by Definition \ref{def: planar kep}. Let $\gamma= \bar u([c,d])$ be a connected component of $\overline{\Xi}_k \cap \bar u([a,b])$, for some $k = 1,\dots,m$. Let $\hat \theta$ be the angle between $\bar u (c)$ and $\bar u(d)$ with respect to the centre $c_k$, as introduced in Definition \ref{def: angle planar}. Note that $\bar u|_{[c,d]} \in K_e$ or $\bar u|_{[c,d]} \in K_o$. 


\begin{theorem}\label{thm: main 3}
In the previous setting, let $\a_k$ be defined by Definition \ref{def: planar kep}. Assume that $\a_k \in [1,2)$ for every $k =1,\dots,m$, and that $\bar u|_{[c,d]}$ is a local minimizer of $\mathcal A_{[c,d]}$ or of $\mathcal M_h$ (for some $h \in \R$) in $K_e$ or in $K_o$. It holds:
\begin{itemize}
\item[($i$)] if $\a_k \in (1,2)$, then $\bar u$ has no self-intersection and no collision in $c_k$ in $(c,d)$;
\item[($ii$)] if $\a_k=1$ and $\bar u (c) \neq \bar u (d)$, then $\bar u$ has no self-intersection and no collision in $c_k$ in $(c,d)$;
\item[($iii$)] if $\a_k =1$, $\bar u(c) = \bar u(d)$ and $u|_{[c,d]} \in K_o$, then one of the following alternative occurs:
\begin{itemize}
\item[($a$)] $\bar u$ has no self-intersection and no collision in $c_k$ in $(c,d)$;
\item[($b$)] $\bar u|_{[c,d]}$ is a collision-ejection minimizer, with a unique collision within the time interval $[c,d]$;
\end{itemize}
\item[($iv$)] if $\a_k =1$, $\bar u(c) = \bar u(d)$, $u|_{[c,d]} \in K_e$, and it is non-constant, then one of the following alternative occurs:
\begin{itemize}
\item[($a$)] $\bar u$ has no self-intersection and no collision in $c_k$ in $(c,d)$;
\item[($b$)] $\bar u|_{[c,d]}$ is a collision-ejection minimizer, with a unique collision within the time interval $[c,d]$.
\end{itemize}
\end{itemize} 
\end{theorem}

\begin{remark}
In case $\bar u(c) = \bar u(d)$ and $\bar u \in K_e$ we have to require that $\bar u$ is not constant, because otherwise it is possible that $\bar u(t) \equiv \bar u(c)$. This is particularly evident if $\bar u$ is a local minimizer of the Maupertuis' functional: if $\bar u(t) \equiv \bar u(c)$, then $\dot{\bar u}(t) \equiv 0$; hence, $\mathcal{M}_h(\bar u)=0$, which is the minimum value in $K_e$, provided $\Xi_k$ has been chosen sufficiently small. On the other hand, note that, even if $\bar u(c)= \bar u (d)$, any $\bar u \in K_o$ cannot be constant.
\end{remark}

\subsection{Strategy of the proofs} 

The main difference between dealing with the action or with the Maupertuis' functional consists in the fact that, while a collision-free minimizer of the action functional is a classical solution of the motion equation, a collision-free minimizer of the Maupertuis' functional has to be suitably re-parametrized in order to give a classical solution (we refer to the appendix). For this reason, the fixed energy case is slightly more complicated, and we give the proof for it.

We start with the proof of Theorem \ref{thm: main 1}, having in mind that a lot of intermediate results hold true also in the planar case. This follows from the local cylindrical symmetry of the potential in $\R^3$, see Definition \ref{def: local interaction} and the subsequent remark. 

For a fixed $k$, let $\gamma = \bar q([c,d])$ be a connected component of $\overline{\Xi}_k \cap \bar q([a,b])$. Let $\hat \theta$ denote the angle between $\bar q(c)$ and $\bar q(d)$ with respect to the axis $r_k$; let us assume that $\bar q|_{[c,d]} \in K_l$, with $\hat \theta + 2\pi l \in \left(0, 2\pi/(2-\alpha_k)\right)$. We have already noted that there is a splitting $\R^3 = \Pi_k \oplus r_k$. Let us consider a system of cylindrical coordinates in $\R^3$, as described in the introduction:
\[
\R^3 \ni q=(u,z) \in \R^2 \times \R \quad  \text{and} \quad u = \rho \exp\{i \theta\}, \quad \text{where} \quad (\rho,\theta) \in \R^+ \times \R / 2\pi \Z.
\]
As already observed, we can choose the frame of reference so that $\Pi_k = \{z=0\}$, $r_k=\{u=0\}$, and the angular component of the projection of $\bar q(c)$ of the plane $\Pi_k$ is $0$. Since we are assuming that the interaction between $\bar q$ and $V$ is locally axially Keplerian, there exists a cylindrical neighbourhood of $\gamma$ in which
\[
V(q) = \frac{m_k}{\alpha_k |u|^{\alpha_k}} + V_0(|u|,z),
\]
with $V_0$ smooth and bounded. It is not difficult to check that, since $\bar q|_{[c,d]}$ is a local minimizer of $\mathcal{M}_h$, it is a solution of the differential equation $\omega^2 \ddot{q}= \n V(q)$ for a certain $\omega \in \R$ in each open interval where $\bar u(t) \neq 0$. Moreover, we can show that the angular momentum of $\bar q|_{[c,d]}$ with respect to the axis $\{u=0\}$, defined by $\mathcal{C}_{\bar q}(t) = \bar u(t) \land \dot{\bar u}(t)$, is constant. 

We assume by contradiction that $\bar q$ has a collision at $t_1 \in (c,d)$, that is, $\bar u(t_1)=0$. We prove that, if a collision occurs, then the angular momentum has to be identically $0$, so that the angular component of $\bar q|_{[c,d]}$ is piecewise constant. Moreover, it is possible to show that the collisions are isolated, so that there exists a sub-interval $[c',d'] \subset [c,d]$ such that $t_1$ is the unique collision-time in $[c,d]$, and the restriction $\bar q|_{[c',d']}$ is minimal for $\mathcal A_{[c',d']}$ (or for $\mathcal{M}_h$). In order to describe the topological behaviour of the path parametrized by $\bar q|_{[c',d']}$ with respect to the axis $\{u=0\}$, we show that $\bar q|_{[c',d']}$ is the limit, in the weak topology of $H^1$, of a sequence of minimizers of some auxiliary problems, the so-called \emph{obstacle problems}: let $d(\eps)$ be the minimum of the action in the set of functions $q=(u,z) \in H^1\left([c,d],\R^3\right)$ such that
\[
\min_{t \in [c',d']} |u(t)|= \eps.
\] 
We show that the function $\eps \mapsto d(\eps)$ is well defined for $\eps$ sufficiently small, is continuous in $0$, and that given $\eps_n \to 0$, the sequence of the minimizers $q_n$ for $d(\eps_n)$ is convergent, in the weak topology of $H^1$, to $\bar q$. With a blow-up analysis, we obtain a description of the topological features of $q_n$ with respect to the axis $\{u=0\}$, proving in particular that, in the limit for $t \to t_1$, up to a suitable scaling the sequence $(q_n)$ converges to a parabolic solution of the $\a_k$-Kepler problem. This allows us to say that, if a collision occurs, necessarily the collision minimizer describes an angle greater than or equal to $2\pi/(2-\a_k)$, which gives a contradiction with the fact that $\bar \t+2l \pi \in \left(0,2\pi/(2-\a_k)\right)$. We mention that the blow-up analysis shares some similarities with the approach adopted by Tanaka in \cite{Ta1,Ta2}. 

In order to deduce this contradiction, we strongly use the fact that the angular component is constant. This is not true in the planar case, so to complete the proof of Theorem \ref{thm: main 3} requires some extra work. In particular, for the proof of Theorem \ref{thm: main 3}, we have to assume that $\a_k \in [1,2)$ for every $k$. If $\a_k \in (1,2)$ we obtain the same contradiction as before, while in case $\a_k=1$, keeping track of the topological information we collected in the blow-up analysis, by means of a Levi-Civita regularization we deduce that, if $\bar q|_{[c,d]}$ has a collision, then it is a collision-ejection minimizer.

Before proceeding with the proofs of our main results, we present, as announced, two applications. 

\section{Applications}\label{sec: appl}

\subsection{An application of Theorem \ref{thm: main 1}.} Let us consider the potential 
\begin{equation}\label{reduced potential}
V(q) = \frac{m}{\alpha |u|^\alpha} + V_0(|u|,z)
\end{equation}
where $m>0$, $\alpha \in (0,2)$, $V_0 \in \mathcal{C}^1\left(\R^3 \setminus \{0\}\right)$, but $V_0(u,z) \to +\infty$ as $(u,z) \to (0,0)$. For $\tau>0$ and $\phi \in \R$, let 
\[
\hat \Lambda_\phi:= \left\{q=(\rho \exp\{i \t\},z) \in H^1([a,b],\R^3)\left| \begin{array}{l}
\rho(t+\tau)=\rho(t), \\
\theta(t+\tau)=\theta(t) + \phi, \\
\rho(t) \neq 0 \text{ for every $t \in [0,\tau]$}  
\end{array} \right.\right\},
\]
and let $\Lambda_\phi$ be its closure with respect to the weak topology of $H^1$. We search for some conditions in order to deduce that a minimizer of the action or of the Maupertuis' functional in (a weakly closed subset of) $\Lambda_\phi$ is collision-free, that is, it belongs to $\hat{\Lambda}_\phi$. Let us assume that $q_{\min}=(\rho_{\min} \exp\{i \t_{\min}\},z_{\min})$ is a minimizer of the action functional $\mathcal{A}_{[0,\tau]}$ in a weakly closed subset of $\Lambda_\phi$.

\begin{theorem}
If $\phi \in (0,2\pi/(2-\alpha))$ and $q_{\min}(t) \neq 0$ for every $t \in [0,T]$, then $q_{\min}$ is collision-free.
\end{theorem}
\begin{remark}
Theorem 12 in \cite{TeVe} is a particular case of this statement. In general, searching for solutions of the $N$-body problem, one can impose some symmetry in order to reduce the general problem to a simpler one. This reduction modifies the shape of the original potential, and one can easily check that, if we consider the hip-hop symmetry studied in \cite{TeVe}, we obtain exactly a reduced potential of type \eqref{reduced potential}; in the expression of $V$, the term $m/(\alpha |u|^\alpha)$ comes from the partial collisions, while $V_0$ comes from the total collisions; our assumption $q_{\min}(t) \neq 0$ reflects the fact that in the quoted paper the authors proved firstly that minimizers of the action functional are free of total collisions, and then turn to the question of partial ones, proving Theorem 12. 
\end{remark}

\begin{proof}
The function $q_{\min}$ has constant angular momentum in $[0,T]$: indeed, this follows by the extremality of $q_{\min}$ with respect to variation of the angular component keeping the radial and the $z$ component fixed: for any $\varphi \in \mathcal{C}^\infty_c((0,\tau))$, let $q_\lambda:=(\rho_{\min} \exp\{i (\t_{\min}+\lambda \varphi)\},z_{\min})$; we have
\[
\left. \frac{d}{d\lambda}\mathcal{A}_{[0,T]}(q_\lambda)\right|_{\lambda=0} = 0,
\]
and since this holds true for any $\varphi\in \mathcal{C}^\infty_c((0,\tau))$, one can easily deduce that $\rho_{\min}^2 \dot{\t}_{\min}= const$. Assume by contradiction that $q_{\min}$ has at least one collision. Arguing as in the forthcoming Subsection \ref{sub: basic 3d}, we can show that the set of collision times of $q_{\min}$ is discrete and finite, and the angular component of $q_{\min}$ is piecewise constant: if, $t_1,\dots,t_m$ denote the collision times of $q_{\min}$, then for every $j$
\[
\t_{\min}(t)= 0  \quad t \in [0,t_1),  \quad \tilde \t(t)= \hat \t \quad t \in (t_m,1] \quad
\tilde \t(t)=const.=\bar \theta_j  \quad t \in (t_j,t_{j+1}).
\]
Moreover, since the potential is locally independent of the angle, it is possible if necessary to replace $q_{\min}$ with another minimizer such that
\beq\label{change the angle2}
0 < \bar \t_{j+1}-\bar \t_j \le \hat \t + 2l\pi;
\eeq
for the reader's convenience, we develop the details of this line of reasoning in the proof of the \eqref{change the angle}.

Let $\bar r:= \min_t |q_{\min}(t)|>0$.  Let us consider 
\[
\Xi:= \left\{ x \in \R^3: \dist\left(x,\{u=0\}\right)<1\right\}.
\]
Let $\gamma = q_{\min}([c,d])$ be any connected component of $\overline{\Xi} \cap q_{\min}([0,\tau])$. For instance, let us consider $[c,d]$ such that $t_1 \in [c,d]$. There exists a cylindrical neighbourhood $\Gamma$ of $\gamma$ such that $\Gamma \subset \Xi \cap \left\{|z|>\bar r/2\right\}$, so that here the potential $V_0$ is smooth and bounded. This means that the interaction between $q_{\min}|_{[c,d]}$ and the singular set $\{u=0\}$ is locally axially Keplerian. Furthermore, thanks to the conservation of the angular momentum and to the \eqref{change the angle2}, the rotation number of $q_{\min}|_{[c,d]}$ (as introduced in Definition \ref{def: winfding number}) belongs to $(0,2\pi/(2-\alpha))$, so that Theorem \ref{thm: main 1} implies that $q_{\min}$ is collision-free in $[c,d]$, a contradiction.
\end{proof}

\subsection{Fixed ends trajectories for a generalized planar $N$-centre problem}

Here we give some applications of Theorem \ref{thm: main 3} to some  planar problems of $N$-centre type, where for each centre we consider a possibly different degree of homogeneity. We consider the study of the motion of a test particle under the Newtonian-like attraction of $N$ fixed heavy bodies, the centres of the problem. Let $c_k \in \R^2$ be the position of the $k$-th centre. We consider the differential equation
\begin{equation}\label{eq appl 2D}
\ddot{u}= -\sum_{k=1}^N \frac{m_k}{|u-c_k|^{2+\alpha_k}}(u-c_k)+ \nabla V_0(u)= \nabla V(u),
\end{equation}
where $m_k>0$ and $\alpha_k \in [1,2)$ for every $k$, 
\[
V(u) = \sum_{k=1}^N\frac{m_k}{\alpha_k|u-c_k|^{\alpha_k}}+ V_0(u),
\]
and $V_0 \in \mathcal{C}^1(\R^2) \cap W^{1,\infty}(\R^2)$ and is such that equation \eqref{eq appl 2D} does not admit stationary solutions. We wish to prove the existence and the multiplicity of solutions of the fixed ends and fixed energy problem
\begin{equation}\label{fixed ends fixed energy}
\begin{cases}
\ddot{u}= \nabla V(u) & \text{in $(a,b)$}\\
\frac{1}{2}|\dot{u}|^2-V(u)=h   & \text{in $(a,b)$}\\
u(a)=p_1 \qquad u(b)=p_2,
\end{cases}
\end{equation}
where $a,b$ are not assigned and $p_1,p_2 \in \R^2 \setminus \{c_1,\dots,c_N\}$. We consider the case $h \ge 0$; for the case $h<0$, we refer to Theorem 1.2 of \cite{SoTe}.

Let 
\[
\wh{H}=\wh{H}_{p_1 p_2}([0,1]):= \left\{ u \in H^1([0,1],\R^2): u(0)=p_1, \ u(1) = p_2, \ u(t) \neq c_j \ \forall t \in [0,1] \right\},
\]
and let $H= H_{p_1 p_2}([a,b])$ be its closure with respect to the weak topology of $H^1$. Since we are dealing with a fixed energy problem, we consider the Maupertuis' functional $\mathcal{M}_h: H \to \R \cup \{+\infty\}$ defined by 
\[
\mathcal{M}_h(u):= \frac{1}{2}\int_0^1 |\dot{u}(t)|^2\,dt \int_0^1 \left(V(u(t))+h \right)\,dt.
\]
As explained in \cite{SoTe}, the paths in $\wh{H}$ can be classified with respect to their winding numbers with respect to the centres. These rotation numbers can be computed by artificially closing the paths of $\wh{H}$, in the following way. Let us fix a smooth function $v \in H^1([1,2],\R^2)$ such that $v(1)=p_2$, $v(2)=p_1$, and $v(t) \neq c_k$ for every $t$. For any $u \in \wh{H}$, we define
\[
\gamma_u(t):= \begin{cases} u(t) & t \in [0,1] \\ v(t) & t \in (1,2]. \end{cases}
\] 
Since $\gamma_u$ is closed, it is well defined the usual winding number
\[
\ind(\gamma_u,c_k) = \frac{1}{2\pi i} \int_{\gamma_u} \frac{dz}{z-c_k},
\]
where $z$ denotes a complex variable. We consider open subsets of $\wh{H}$ of type
\[
\wh{\mathfrak{H}}_{(l_1,\dots,l_N)}=\left\{u \in \wh{H}: \ind(\gamma_u,c_k)=l_k \quad \forall k=1,\dots,N\right\},
\]
where $(l_1,\dots,l_N) \in \Z^N$. We collect together all the classes having winding number with the same parity with respect to each centre, that is, for a given $(l_1\dots,l_N) \in \Z_2^N$, we define the open set
\[
\wh{H}_{(l_1,\dots,l_N)}:= \left\{u \in \wh{H}: \ind(\gamma_u,c_k) \equiv l_k \mod 2,  \quad \forall k=1,\dots,N\right\},
\]
and we consider its closure $H_{(l_1,\dots,l_N)}$ with respect to the weak topology of $H^1$.

\begin{theorem}
Let $(l_1,\dots,l_N) \in \Z_2^N$ such that
\begin{equation}\label{hstrana}
\exists \ i \neq j: \ l_i \not \equiv l_j \mod 2.
\end{equation}
There exists a weak solution $x_{(l_1,\dots,l_N)}$ of problem \eqref{fixed ends fixed energy} which is a re-parametrization of a minimizer $u_{(l_1,\dots,l_N)}$ of the Maupertuis' functional $\mathcal{M}_h$ in $H_{(l_1,\dots,l_N)}$. In particular:
\begin{itemize}
\item[($i$)] if there exists $k$ such that $\alpha_k=1$, and
\begin{equation}\label{cond necessaria coll 1}
l_k \not \equiv l_j \mod 2 \qquad \forall j \neq k,
\end{equation}
then either $u_{(l_1,\dots,l_N)}$ is collision-free, or it is a collision-ejection minimizer, with precisely one collision in $c_k$ within the time interval $[0,1]$;
\item[($ii$)] if there exist $j \neq k \in \{1,\dots,N\}$ such that $\alpha_k=\alpha_j=1$,
\begin{equation}\label{cond necessaria coll 2}
l_k \equiv l_j \mod 2 \quad \text{and} \quad l_k \not \equiv l_m \mod 2 \quad \forall m \not \in \{k,j\}, 
\end{equation}
then either $u_{(l_1,\dots,l_N)}$ is collision-free, or it is a collision-ejection minimizer, with precisely one collision in $c_k$ and one collision in $c_j$ within the time interval $[0,1]$;
\item[($iii$)] in all the other cases, $u_{(l_1,\dots,l_N)}$ is collision-free.
%
%
%
%
\end{itemize}
\end{theorem}
The following pictures represent, respectively, a collision-ejection minimizer in $H_{(0,0,1,0,0)}$ with a unique collision, a collision-ejection minimizer in $H_{(1,0,1,0,0)}$ with two collisions within $[0,1]$, and a collision-free minimizer in $H_{(1,1,0,0,1)}$.
\begin{center}
\begin{tikzpicture}[>=stealth, scale=1.5]
\filldraw (0.3,-0.6) circle (1pt)
          (0.8,0.3) circle (1pt)
          (0,0.7) circle (1pt)
          (-0.6,-0.3) circle (1pt)
          (-0.5,0.1) circle (1pt) 
          (-1.5,1.5) circle (1pt)
          (-0.7,0.83) circle (1pt);
\draw[<<->] (-1.5,1.5) .. controls (0,0.3) and (0,0) .. (0.3,-0.6);
\draw[font=\footnotesize] (0.3,-0.8) node{$c_3$};
\draw[font=\footnotesize] (0.6,0.3) node{$c_2$};
\draw[font=\footnotesize] (-0.2,0.7) node{$c_1$};
\draw[font=\footnotesize] (-0.6,-0.5) node{$c_4$};
\draw[font=\footnotesize] (-0.5,-0.1) node{$c_5$};
\draw[font=\footnotesize] (-1.2,1.5) node{$p_2$};
\draw[font=\footnotesize] (-0.7,1) node{$p_1$};
\end{tikzpicture}
$\qquad \qquad $
\begin{tikzpicture}[>=stealth, scale=1.5]
\filldraw (0.3,-0.8) circle (1pt)
          (0.9,0.6) circle (1pt)
          (-0.2,1.3) circle (1pt)
          (-0.8,-0.9) circle (1pt)
          (-0.9,0) circle (1pt) 
          (0.13,-0.2) circle (1pt)
          (-0.03,0.3) circle (1pt);
\draw[<<->] (0.3,-0.8) .. controls (0,0.3) and (0,0) .. (-0.2,1.3);
\draw[font=\footnotesize] (0.5,-0.6) node{$c_3$};
\draw[font=\footnotesize] (0.9,0.4) node{$c_2$};
\draw[font=\footnotesize] (0,1.1) node{$c_1$};
\draw[font=\footnotesize] (-0.8,-0.7) node{$c_4$};
\draw[font=\footnotesize] (-0.9,0.2) node{$c_5$};
\draw[font=\footnotesize] (0.3,-0.2) node{$p_2$};
\draw[font=\footnotesize] (-0.2,0.3) node{$p_1$};
\end{tikzpicture}
$\qquad \qquad$
\begin{tikzpicture}[>=stealth, scale=1.5]
\filldraw (0.3,-0.6) circle (1pt)
          (0.8,0.3) circle (1pt)
          (0,0.7) circle (1pt)
          (-0.6,-0.3) circle (1pt)
          (-0.5,0.5) circle (1pt) 
          (1.5,-0.2) circle (1pt)
          (-0.7,0.83) circle (1pt);
\draw[->] (-0.7,0.83) .. controls (-0.5,-0.3) and (0,-0.2) .. (1.5,-0.2);
\draw[font=\footnotesize] (0.3,-0.8) node{$c_3$};
\draw[font=\footnotesize] (0.6,0.3) node{$c_2$};
\draw[font=\footnotesize] (-0.2,0.7) node{$c_1$};
\draw[font=\footnotesize] (-0.6,-0.5) node{$c_4$};
\draw[font=\footnotesize] (-0.4,0.3) node{$c_5$};
\draw[font=\footnotesize] (1.5,0) node{$p_2$};
\draw[font=\footnotesize] (-0.7,1) node{$p_1$};
\end{tikzpicture}

\end{center}

\begin{remark}
Assumption \eqref{hstrana} permits to exclude degenerate situations, such as the fact that a minimizer is constant.
\end{remark}

\begin{proof}
Let $(l_1,\dots,l_N) \in \Z_2^N$ such that \eqref{hstrana} holds. It is not difficult to prove that the functional $\mathcal{M}_h$ is weakly lower semi-continuous and coercive, so that there exists a minimizer $\bar u=u_{(l_1,\dots,l_N)}$ in the weakly closed set $H_{(l_1,\dots,l_N)}$. We assume that $\bar u$ has some collisions, and we show that necessarily we are in cases ($i$) or ($ii$) of the theorem. Let us introduce $N$ neighbourhoods
\[
\Xi_k :=\left\{ u \in \R^2: |u-c_k|<d_k\right\},
\]
where $d_k$ is chosen in such a way that $\Xi_j\cap \Xi_k = \emptyset$ of $j \neq k$, and $p_1,p_2 \not \in \Xi_k$ for every $k$. Using these neighbourhoods in Definition \ref{def: planar kep}, from the explicit expression of the potential $V$ we deduce that the interaction between $\bar u$ and $\{c_1,\dots,c_N\}$ is locally Keplerian. Let $t_1$ be the first collision time of $\bar u$; let us say that $\bar u(t_1)=c_k$, and let $\bar u([c,d])$ be the connected component of $\bar u([0,1]) \cap \Xi_k$ containing $t_1$. Since $\bar u$ minimizes $\mathcal{M}_h$ in $H_{(l_1,\dots,l_N)}$, the restriction $\bar u|_{[c,d]}$ minimizes $\mathcal{M}_h$ in one set between $K_e$ and $K_o$, which have been defined in Subsection \ref{sub: intro planar}. Even if $\bar u$ in $K_e$, it cannot be constant in $[c,d]$: indeed, $\bar u(c) \in \pa \Xi_k$ and $\bar u(t_1)=c_k$. We are then in position to apply Theorem \ref{thm: main 3}, deducing that $\bar u|_{[c,d]}$ is a collision ejection minimizer, and $t_1$ is the unique collision time of $\bar u$ in $[c,d]$. Now, $\bar u$ solves equation 
\begin{equation}\label{eq per bar u applica}
\bar \omega^2 \ddot{\bar u}= \nabla V(\bar u) \quad \text{where} \quad \bar \omega^2 = \frac{\int_0^1 V(\bar u) +h}{\frac{1}{2} \int_0^1 |\dot{\bar u}|^2 },
\end{equation} 
in $(c,t_1)$ and in $(t_1,d)$, satisfies the collision-ejection condition in $[c,d]$, and this equation is reversible with respect to the time involution $t \mapsto -t$; as a consequence, the uniqueness theorem for the initial value problems ensures that $\bar u(t_1-t) = \bar u(t_1+t)$ whenever both $t_1-t$ and $t_1+t$ are in $[0,1]$. This means that $\bar u$ is an ejection-collision minimizer. 

If $t_1$ is the unique collision time of $\bar u$, then necessarily we are in case ($i$) or ($ii$) of the theorem; this follows directly from the definition of the component $\wh{H}_{(l_1, \dots, l_N)}$. If there exists another collision time $t_2$, then repeating the previous line of reasoning we deduce that $\bar u(t_2-t) = \bar u(t_2+t)$ whenever both $t_2-t$ and $t_2+t$ are in $[0,1]$; this implies that the trajectory of $\bar u$ bounces between two centres $c_k$ and $c_j$, and $p_1$ and $p_2$ belongs to this trajectory. Again, by definition of $\wh{H}_{(l_1,\dots,l_N)}$, it follows that we are in case ($iii$) of the theorem. The fact that in the time interval $[0,1]$ the function $\bar u$ has only two collision times follows from the minimality of $\bar u$.
\end{proof}

\section{Proof of Theorem \ref{thm: main 1}}

\subsection{Basic properties of a local minimizer}\label{sub: basic 3d}

For a fixed $k$, let $\gamma = \bar q([c,d])$ be a connected component of $\overline{\Xi}_k \cap \bar q([a,b])$. We assume that $[c,d]=[0,1]$ to simplify the notation. Let $\hat \theta$ denote the angle between $\bar q(0)$ and $\bar q(1)$ with respect to the axis $r_k$; we are assuming that $\bar q|_{[0,1]} \in K_l$, with $\hat \theta + 2\pi l \in \left(0, 2\pi/(2-\alpha_k)\right)$.

As explained above, it is possible to introduce a frame of reference $(u,z) = (\rho \exp\{i \theta\},z) \in \R^2 \times \R$ in such a way that 
\[
V(q) = \frac{m_k}{\alpha_k |u|^{\alpha_k}} + V_0(|u|,z) 
\]
for $q$ in a cylindrical neighbourhood $\Gamma$ of $\gamma$; we recall that $m_k>0$ and $V_0 \in \mathcal{C}^1\left(\overline{\Gamma}\right)$. Since we proceed with a local argument, we write $m$ instead of $m_k$ and $\alpha$ instead of $\alpha_k$ to simplify the notation. 

We assume that $\bar q$ has a collision in $(0,1)$, and we wish to show that this implies $\hat \theta > 2\pi/(2-\alpha)$, in contradiction with our assumption. In what follows, we write 
\[
\bar q= (\bar u,\bar z) = (\bar \rho \exp\{i \bar \theta\},\bar z).
\]
As already announced, the proof of Theorem \ref{thm: main 1} requires a lot of intermediate results which we use also in the proof of Theorem \ref{thm: main 3}. For this reason, we explicitly point out when we use the fact that, in the present problem, the term $V_0$ in Definition \ref{def: local interaction} depends on $u$ only through $|u|$, while in the planar case $V_0$ this is not true. Clearly, those statements have to be neglected in the next section.

\medskip


As $\bar q(0) \in \pa \Xi_k$ and $\bar q(t_1) \in r_k$ for some $t_1 \in (0,1)$, $\bar q$ is not constant and $\|\dot{\bar q}\|_2>0$. As a consequence, it is well defined the quantity
\[
\bar \omega^2:= \frac{\int_0^1 (V(\bar q)+h)}{\frac{1}{2} \int_0^1 \dot{\bar q}^2}.
\]
By minimality, we know that (see Proposition \ref{conservazione dell'energia}) the energy function
\[
t \mapsto \frac{1}{2}|\dot{\bar q}(t)|^2 - \frac{V(\bar q(t))}{\bar \omega^2},
\]
which is defined almost everywhere in $[0,1]$, is constant and equal to $h/\bar \omega^2$. In particular, this implies that 
\begin{equation}\label{oss su omega}
\bar \omega^2 = \frac{\int_c^d (V(\bar q)+h)}{\frac{1}{2} \int_c^d \dot{\bar q}^2} \quad \text{for every $(c,d) \subset [0,1]$}.
\end{equation}
Let $T_c(\bar q)$ be the set of the collision times of $\bar q$:
\[
T_c(\bar q) = \left\{ t \in [0,1]: \bar q(t) \in \Sigma\right\}.
\]

\begin{lemma}
It results $T_c(\bar q)= \left\{ t \in [0,1]: \bar q(t) \in r_k\right\}$.
\end{lemma}
\begin{proof}
It is not obvious, because we are not assuming that $r_i \cap r_j = \emptyset$ whenever $i \neq j$. On the other hand, by Definition \ref{def: local interaction}, it is immediate to observe that $\Sigma \cap \Gamma \subset r_k$ (where $\Gamma$ is the cylindrical neighbourhood $\Gamma$ of $\gamma=\bar q([0,1])$, given by Definition \ref{def: local interaction}). Indeed, if this is not true, there exists $j \neq k$ such that $r_j \cap r_k \cap \Gamma \neq \emptyset$. Thus, by the expression of $V$ in $\Gamma$, we see that there exists $p=(u^k,z^k) \in r_j \setminus r_k$ such that 
\[
V(p) = \frac{m}{\alpha |u^k|^\alpha} + V_0(|u^k|,z^k) < +\infty,
\] 
in contradiction with the fact that $p$ is a singular point of $V$.
\end{proof}

Since $\bar q$ is continuous and $\mathcal{M}_h(\bar q)<+\infty$, the set $T_c(\bar q)$ is a closed set of null measure, and its complement is the union of a finite or countable number of closed intervals. It is well known that, if $(c,d) \subset [0,1] \setminus T_c(\bar q)$, then $\bar q$ is a smooth solution of a differential equation in $(c,d)$ (see Theorem \ref{teorema 4.1}). A remarkable fact is that, even if $\bar q$ has some collisions, its $z$ component is smooth. 

\begin{lemma}\label{smoothness of z}
Let $(c,d) \subset [0,1]$ be a connected component of $(c,d) \setminus T_c(\bar q)$. Then $\bar q \in \mathcal{C}^2((c,d))$ and
\beq\label{eq q_k^h}
\bar \o^2 \ddot{\bar q}(t)= \n V(\bar q(t)) \qquad t \in (c,d).
\eeq
Moreover, $\bar z \in \mathcal{C}^2([0,1])$ and
\[
\bar \o^2\ddot{\bar z}(t)= \pa_z V_0(|\bar u(t)|, \bar z(t)) \qquad t \in (0,1).
\]
\end{lemma}
\begin{proof}
The first part is a consequence of the extremality of $\bar q$ with respect to variations with compact support in $(c,d)$: for any $\varphi \in \mathcal{C}^\infty_c(c,d)$, it results $\bar q+\lambda \varphi \in K_l$ for $|\lambda|$ sufficienlty small; as a consequence,
\begin{equation}\label{extremality}
\left.\frac{d}{d\lambda} \mathcal{M}_h(\bar q+\lambda \varphi)\right|_{\lambda=0}=0,
\end{equation}
and a direct computation gives the desired result.

As far as the smoothness of $\bar z$ is concerned, one can consider variations of the $z$ component only, of type $\bar q+\l \f$, with $\f=(0,\zeta) \in \mathcal{C}_c^{\infty}(0,1)$; such variations are in $K_l$, whenever $|\l|$ is not too large. By computing again \eqref{extremality} in the present case, we obtain the desired result thanks to the boundedness of $\pa_z V_0$.
\end{proof}

The following lemma gives a more precise description of the set $T_c(\bar q)$: it is a discrete and finite set, that is, \emph{the collisions are isolated}. This is a generalization of a known fact in our particular setting, see \cite{BaFeTe08, FeTe, SoTe}.

\begin{lemma}\label{collisions isolated}
The set $T_c(\bar q)$ is discrete and has a finite number of elements. 
\end{lemma}

\begin{proof}
Assume by contradiction that $t_0$ is an accumulation point in the set $T_c(\bar q)$. Then there exists a sequence of intervals $((a_n,b_n))$, with $(a_n,b_n) \subset [0,1]$, such that $a_n \to t_0$ and $b_n \to t_0$ as $n \to \infty$, $\bar u(a_n)=0=\bar u(b_n)$ for every $n$, and $|\bar u(t)|>0$ for every $t \in (a_n,b_n)$. Let us set $I(t):= |\bar u(t)|^2$. Since $t \mapsto \bar u(t)$ is a classical solution of \eqref{eq q_k^h} for $t \in (a_n,b_n)$, we can differentiate twice $I$; using the conservation of the energy and the equation of $\bar u$, we obtain the following modified Lagrange-Jacobi identity:
\begin{align*}
\ddot{I} & = 2 |\dot{\bar u}|^2+2 \langle \bar u, \ddot{\bar u} \rangle \\
&=\frac{4h}{\bar \o^2}+\frac{2}{\a \bar \o^2}(2-\a)\frac{m}{|\bar u|^\a}-2|\dot{\bar z}|^2  +\frac{4}{\bar \o^2}V_0(\bar q) + \frac{2}{\bar \o^2}  \partial_\rho V_0(|\bar u|,\bar z) |\bar u|.
\end{align*}
Let $\xi_n \in (a_n,b_n)$ be a maximum point of $I$ in $(a_n,b_n)$; by maximality $\ddot{I}(\xi_n) \leq 0$ for every $n$. On the other hand, since $\a \in (0,2)$, in a neighbourhood of $t_0$ the second term in the expression of $\ddot{I}$ becomes arbitrarily large, while the other terms are bounded (we use the previous lemma to deduce that $\dot{\bar z}$ is bounded in $[0,1]$); therefore
\[
\lim_{n \to \infty} \ddot{I}(\xi_n)=+\infty,
\]
a contradiction. Each collision time is isolated and, by compactness, the interval $[0,1]$ contains only a finite number of them. 
\end{proof}

By the previous proof, we deduce also a useful property of the function $I$ when $q$ approaches the singularity.

\begin{corollary}\label{convexity of I}
If $t_0 \in T_c(\bar q)$, then there exists a neighbourhood of $t_0$ such that the function $I(t)= |\bar u(t)|^2$ is strictly convex.
\end{corollary}

Let us conclude this preliminary subsection with a result which is a consequence of the peculiar form of $V_0$, specified by Definition \ref{def: local interaction}. This means that what follows cannot be taken into account for the proof of Theorem \ref{thm: main 3}.

\begin{lemma}\label{lem: angolo costante}
Let $(c,d) \subset [0,1]$ a connected component of $[0,1] \setminus T_c(\bar q)$. Then $\bar \t$ is piecewise constant in $(c,d)$.
\end{lemma}
\begin{proof}
We show that $\dot{\bar \t}$ vanishes identically in $(c,d)$. Let 
\[
\mathfrak{C}_{\bar q}(t)=|\bar u(t) \land \dot{\bar u}(t)| = \bar{\rho}^2(t) \dot{\bar \t}(t)
\] 
be the angular momentum of $\bar q$ with respect to the $z$ axis. As observed in Remark \ref{rem: su V 3D}, it is constant in $(c,d)$, because therein $\bar q$ solves \eqref{eq q_k^h}. By the conservation of the energy, we have
\[
\frac{\bar \o^2}{2} \dot{\bar \rho}^2+ \frac{\bar \o^2}{2 \bar{\rho}^2} \mathfrak{C}_{\bar q}^2+\frac{\bar \o^2}{2} \dot{\bar{z}}^2 - \frac{m}{\a \bar \rho^\a}- V_0(\bar q)= h
\]
in $(c,d)$, so that
\[
 \frac{\a  \bar 	\o^2 \mathfrak{C}_{\bar q}^2 - 2 m \bar \rho^{2-\a} - V_0(\bar q) \bar \rho^2}{2 \a \bar \rho^2} \le h
\]
in $(c,d)$. A necessary condition in order to satisfy the previous relation when $\bar \rho \to 0$ is that $\mathfrak{C}_{\bar q} = 0$, i.e. $\dot{\bar \t}=0$ in $(c,d)$.
\end{proof}

The previous lemmas imply that, if $\bar q$ is a collision minimizer of $\mathcal{M}_h$, then $T_c(\bar q)=\{t_1,\ldots,t_m\}$ for some $m \in \N$, and (being $t_0=0$ and $t_{m+1}=1$)
\[
\bar \t(t)= \bar{\t}_j \qquad \forall t \in (t_j,t_{j+1}), \ j=0,\ldots,m,
\]
where $\bar \t_0= 0$ and $\bar \t_{m+1}=\hat \t$ (for the reader's convenience, we recall that 
\[
\bar u(0) = \bar \rho(0) \exp\{i 0\} \quad \text{and} \quad \bar u(1) =\bar \rho(1) \exp\{i \hat \theta\}).
\] 
Now, Definition \ref{def: local interaction} says that $V(q)$ is independent on $\t$ (at least in a cylindrical neighbourhood of $\bar q([0,1])$. Therefore, any function $\tilde{q}$ defined by
\[
\begin{cases}
\tilde \rho(t)= \bar \rho(t) \\
\tilde z(t)= \bar z(t) \\
\tilde \t(t)= 0  \quad t \in [0,t_1),  \quad \tilde \t(t)= \hat \t \quad t \in (t_m,1] \\
\tilde \t(t)=const.  \quad t \in (t_j,t_{j+1}), \ j=1,\ldots,m-1
\end{cases}
\]
is a minimizer of $\mathcal{M}_h$ in $K_l$. As a consequence, it is not restrictive to assume that, if $\bar q$ is a collision minimizer of $\mathcal{M}_h$ in $K_l$, it results 
\beq\label{change the angle}
0 < \bar \t_{j+1}-\bar \t_j \le \hat \t + 2l\pi.
\eeq
In the first picture we represent a minimizer $\bar u$ of $\mathcal{M}_h$ in $K_l$ with $l=0$, not satisfying the \eqref{change the angle}; in the second picture we consider another function, $\tilde u$, obtained by $\bar u$ after a change in the angular component, in such a way that the \eqref{change the angle} is fulfilled. As $V$ does not depend on $\theta$, $tilde u$ is another minimizer of $\mathcal{M}_h$ in $K_l$ with $l=0$.
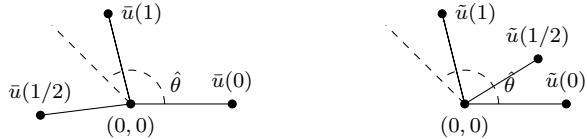
\begin{center}
\begin{tikzpicture}[>=stealth, scale=1.5]
\filldraw (0.0,0.0) circle (1pt)
          (-0.2,0.8) circle (1pt)
          (0.9,0) circle (1pt)
          (-0.8,-0.1) circle (1pt);
\draw (0.9,0)--(0,0)--(-0.2,0.8)--(0,0)--(-0.8,-0.1); 
\draw[dashed] (0,0)--(-0.7,0.7);
\draw[dashed] (3mm,0mm) arc (0:135:3mm);
\draw[font=\footnotesize] (0.4,0.2) node{$\hat \t$};
\draw[font=\footnotesize] (0,-0.2) node{$(0,0)$};
\draw[font=\footnotesize] (0.9,0.2) node{$\bar u(0)$};
\draw[font=\footnotesize] (0.1,0.8) node{$\bar u(1)$};
\draw[font=\footnotesize] (-0.8,0.1) node{$\bar u(1/2)$};
\end{tikzpicture}
$\qquad \qquad$
\begin{tikzpicture}[>=stealth, scale=1.5]
\filldraw (0.0,0.0) circle (1pt)
          (-0.2,0.8) circle (1pt)
          (0.9,0) circle (1pt)
          (0.65,0.4) circle (1pt);
\draw (0.9,0)--(0,0)--(-0.2,0.8)--(0,0)--(0.65,0.4); 
\draw[dashed] (0,0)--(-0.7,0.7);
\draw[dashed] (3mm,0mm) arc (0:135:3mm);
\draw[font=\footnotesize] (0.4,0.2) node{$\hat \t$};
\draw[font=\footnotesize] (0,-0.2) node{$(0,0)$};
\draw[font=\footnotesize] (0.9,0.2) node{$\tilde u(0)$};
\draw[font=\footnotesize] (0.1,0.8) node{$\tilde u(1)$};
\draw[font=\footnotesize] (0.65,0.6) node{$\tilde u(1/2)$};
\end{tikzpicture}
\end{center}

\subsection{The obstacle problems}\label{sub: obstacle}

We pass from a global analysis of the minimizer $\bar q$ to a local study in a neighbourhood of a collision. This is possible because, as shown in Lemma \ref{collisions isolated}, the collisions are isolated: if $\bar q$ has a collision at time $t_1$, then there exist $c,d \in [0,1]$ such that $c<t_1<d$, and $t_1$ is the unique collision time in $[c,d]$; we can choose $c$ and $d$ in such a way that 
\begin{itemize}
\item the function $I=|\bar u|^2$ is strictly convex in $(c,d)$ (see Corollary \ref{convexity of I});
\item $|\bar u(c)|=|\bar u(d)|$; we set $\hat \rho:= |\bar u(c)|$, and choose $c,d$ so that $\hat \rho<d_k$ (where we recall that $d_k$ has been introduced in Definition \ref{def: local interaction}).
\end{itemize} 
We set $\hat \t_1:= \bar \t(d)-\bar \t(c)$; as already observed, it is not restrictive to assume that $0< \hat \t_1 \le \hat \t+2l \pi$. Note that, if $\hat \t_1=2\pi$ (which is admissible for $\alpha \in (1,2)$), this does not exclude that $\bar p_1 = \bar p_2$. 

Let $\Gamma_1$ be a cylindrical neighborhood of $q|_{[c,d]}$, of type
\[
\Gamma_1:= \left\{ (u,z) \in \R^3: |u| < \hat \rho, \ z \in \left(\inf_{t \in [c,d]} \bar z(t) - \delta, \sup_{t \in [c,d]} \bar z(t) + \delta\right) \right\}
\]
for some $\delta>0$. Since $\hat \rho<d_k$, the set $\Gamma_1$ is compactly contained in $\Xi_k$ (introduced in Definition \ref{def: local interaction}).

Let
\[
\wh{\mathcal{K}}:= \left\lbrace q \in H^1\left([c,d]\right) \left| \begin{array}{l}
\text{$|u(t)| \neq 0$ and $q(t) \in \overline{\Gamma}_1$ for every $t \in [c,d]$,}\\
\text{$q(c)=\bar p_1$, $q(d)= \bar p_2$, and the function} \\
\begin{cases} \bar q(t) & t \in [0,c) \cup (d,1] \\ q(t) & t \in [c,d] \end{cases} \text{ belongs to $K_l$,} \\
\end{array}  \right. \right\},
\]
and let $\mathcal{K}$ be its closure with respect to the weak topology of $H^1$ (we remark that $\mathcal{K} \setminus \hat{\mathcal{K}}$ consists in collision functions and functions leaning on the boundary $\pa \Gamma_1$). We consider the restriction of the Maupertuis' functional (still denoted $\mathcal{M}_h$, with some abuse of notation) to $\mathcal{K}$:
\[
\mathcal{M}_h(q)=\frac{1}{2} \int_c^d |\dot{q}(t)|^2\,dt \int_c^d \left( V(q(t))+h\right)\,dt.
\]
\begin{lemma}\label{lem: proprieta' di M_h}
The functional $\mathcal{M}_h$ is weakly lower semi-continuous and coercive in $\mathcal{K}$. Moreover, $\mathcal{M}_h(q) \ge C>0$ for every $q \in \mathcal{K}$. 
\end{lemma}
\begin{proof}
The weak lower semi-continuity of $\mathcal{M}_h$ is standard, once that we noticed that, since the image of the functions of $\mathcal{K}$ is confined in $\Gamma_1 \subset \Xi_k$ (recall that $\Xi_k$ has been introduced in Definition \ref{def: local interaction}), it results
\[
V(q(t)) = \frac{M}{\alpha |u(t)|^{\alpha}} + V_0(u(t),z(t)) \qquad \forall t \in [c,d], 
\]
for every $q=(u,z) \in \mathcal{K}$.

As far as the coercivity is concerned, first of all we observe that
\begin{equation}\label{positività di int pot}
\int_c^d \left( V(q)+h \right) \ge C > 0 \qquad  \forall q \in \mathcal{K};
\end{equation}
this follows from the choice of $\Xi_k$ (see Remark \ref{rem: su V 3D}) and the fact that $\Gamma_1 \subset \subset \Xi_k$. Now, let us consider a sequence $(q_n) \subset \mathcal{K}$ such that $\|q_n\| \to +\infty$. The boundedness of $\Gamma_1$ implies that there exists $C>0$ such that $\|q\|_2 \le C (d-c)$ for every $q \in \mathcal{K}$. Consequently, $\|\dot{q}_n\|_2 \to +\infty$, and thanks to \eqref{positività di int pot} it results $\mathcal{M}_h(q_n) \to +\infty$. 

\end{proof}

Although the functional $\mathcal{M}_h$ is not addictive, by Proposition \ref{localizzazione dei minimi_M} we know that $\bar q|_{[c,d]}$ is a minimizer of $\mathcal{M}_h$ on $\mathcal{K}$. 

\begin{lemma}\label{lem: unique minimum}
The function $\bar q|_{[c,d]}$ is the unique minimizer of $\mathcal{M}_h$ in $\mathcal{K}$.
\end{lemma}
\begin{proof}
It is convenient to argue in terms of the Jacobi length $\mathcal{L}_h$, defined in the appendix, because this functional has the remarkable property of being addictive. 

Since $\bar q|_{[c,d]}$ is a minimizer of $\mathcal{M}_h$ in $\mathcal{K}$, it is a minimizer of $\mathcal{L}_h$ in the same class. To prove the desired result, we have to show that for every $q \in \mathcal{K}$ it results $\mathcal{L}_h(\bar q|_{[c,d]}) \le \mathcal{L}_h(q)$, and the equality holds if and only if $q$ can be obtained by $\bar q|_{[c,d]}$ by means of a re-parametrization.

Before proceeding with the core of the proof, we note that, thanks to our choice of $c,d$, there exists $\delta>0$ such that $t_1$ is the unique collision time of $\bar q$ in $[c-\delta,d+\delta]$. In particular, recalling Lemma \ref{smoothness of z}, we know that $\bar x(t):=\bar q(\bar \omega t)$ is a solution of $\bar \omega^2 \ddot{\bar q} = \nabla V\left(\bar q\right)$ for $t \in \left(c-\delta, t_1\right)$ and for $t \in (t_1, d+\delta)$.

Assume by contradiction that there exists a minimizer $\tilde q \in \mathcal{K}$ of $\mathcal{L}_h$, which cannot be obtained by $\bar q|_{[c,d]}$ by means of a re-parametrization. As $\mathcal{L}_h$ is addictive, the function
\[
\hat q(t):=\begin{cases}
\bar q(t) & t \in [0,c) \cap (d,1] \\
\tilde q(t) & t \in [c,d]
\end{cases}
\]
is a minimizer of $\mathcal{L}_h$ in $K_l$. Since $|\tilde u(c)|= \hat \rho$, there exists $0<\delta_1<t_1$ such that $\hat q(t) \not \in \Sigma$ for every $t \in [c-\delta,c+\delta_1]$. As illustrated in the appendix, it is then possible to re-parametrize $\hat q$ in order to obtain a smooth solution $\hat x$ of the equation $\ddot{x}=\nabla V(x)$ in a time interval $(\tau_1,\tau_2)$, such that
\[
\hat x(\tau_1) = \hat q(c-\delta)= \bar q(c-\delta) \quad \text{and} \quad \hat x(\tau_2) = \hat q(c+\delta_1)= \tilde q(c+\delta_1),
\]
and there exists $\tau_3 \in (\tau_1,\tau_2)$ such that $\hat x(\tau_3) = \bar q(c)$. By construction, $\hat x$ has to coincide with $\bar x$ in the time interval $(\tau_1,\tau_3]$, because $\hat q \equiv \bar q$ in $(c-\delta,c]$). As a consequence, $\dot{\hat x}(c^-) = \dot{\bar x}(c^-)$, and recalling that both $\hat x$ and $\bar x$ are regular in $c$, and that $\hat x(c) = \bar x(c)$, the uniqueness theorem for the initial value problems ensures that $\hat x \equiv \bar x$ until one of them met the singular set $\{u=0\}$; this implies that there exists a time interval $[c-\delta, \tilde t_1]$ such that $\hat q|_{[c-\delta, \tilde t_1]}$ can be obtained by $\bar q|_{[c-\delta,t_1]}$ with a re-parametrization. An analogue line of reasoning permits to deduce that there exists a time interval $[\tilde t_2,d+\delta]$, with $\tilde t_2 \ge \tilde t_1$, such that $\hat q|_{[\tilde t_2,d+\delta]}$ can be obtained by $\bar q|_{[t_1,d+\delta]}$ with a re-parametrization. If $\tilde t_2> \tilde t_1$ and $\tilde q$ does not rest in $\{u=0\}$ in $(\tilde t_1,\tilde t_2)$, then $\mathcal{L}_h(\tilde q) > \mathcal{L}_h(\bar q)$, in contradiction with the minimality of $\tilde q$; otherwise $\tilde q$ can be obtained by $\bar q|_{[c,d]}$ by means of a re-parametrization, in contradiction with our assumption.
\end{proof}

We introduce the subsets
\[
\mathcal{K}_{\eps}:= \left\{ q=(u,z) \in \mathcal{K}:  \min_{t \in [c,d]} |u(t)|=\eps \right\},
\]
and we consider the following \emph{obstacle problems}: for $\eps \geq 0$, we seek for minimizers of
\[
d(\eps):= \inf \left\lbrace \mathcal{M}_{h}(q) : q \in \mathcal{K}_\eps \right\rbrace .
\]
The value $d(0)$ is the infimum of $\mathcal{M}_h$ on the collision elements of $\mathcal{K}$, and, by assumption, is achieved by $\bar q|_{[c,d]}$. Furthermore, by Lemma \ref{lem: proprieta' di M_h}, and noting that $\mathcal{K}_{\eps}$ is closed in the weak topology of $H^1$, we deduce that the value $d(\eps)$ is achieved by a function $q_\eps \in \mathcal{K}_{\eps}$. We point out that for $\eps$ sufficiently small any function in $\mathcal{K}_\eps$ cannot be constant. Therefore, $d(\eps)>0$.

\begin{lemma}\label{lemma 2.6}
The function $\eps \mapsto d(\eps)$ is continuous in $\eps=0$.
\end{lemma}

\begin{proof} 
\textbf{Step 1)} \emph{It results
 \begin{equation}\label{continuita di d, step 1}
 \limsup_{\eps \to 0^+} d(\eps) \leq d(0).
 \end{equation}}
We know that $t_1$ is the unique collision time of $\bar q$ in $[c,d]$; since in $\Gamma_1$ the expression of $V$ is given by Definition \ref{def: local interaction}, the classical asymptotic estimates (see e.g. \cite{BaFeTe08}) for collision trajectories of the $\alpha$-Kepler' problem applies:
\beq\label{eq300}
\begin{cases}
|\bar u(t)| \simeq C|t-t_1|^{\frac{2}{\a+2}} \\
|\dot{\bar u}(t) |\simeq C|t-t_1|^{-\frac{\a}{\a+2}}
\end{cases} \qquad \text{as $t \to t_1$}.
 \eeq
For $\eps>0$ sufficiently small, let $\zeta_{\pm}(\eps)$ be positive solutions (uniquely determined by the convexity of the function $I(t)=|\bar u(t)|^2$ in $[c,d]$) of
 \[
 |\bar u(t_1+\zeta_+(\eps))|=\eps, \qquad |\bar u(t_1-\zeta_-(\eps))|=\eps.
 \] 
By the estimates \eqref{eq300}, we infer 
\beq\label{eq301}
\zeta_\pm(\eps) \simeq C\eps^{\frac{\a+2}{2}}.
\eeq
Let $\theta_\pm \in \R$ be defined by the relations $u(t_1 \pm \zeta_{\pm}(\eps))=\eps \exp\{i \t_\pm\}$, $\t_+-\t_- = \hat \t_1$. We consider the following variation of $\bar q$:
\[
q_\eps(t):= \begin{cases} 
                    \bar q(t) & t \in [c,d] \setminus T_\eps(\bar q) \\
                    (\eps \exp\{i \t(t)\}, \bar z(t)) & t \in T_\eps(\bar q),
                   \end{cases}
\]
where $T_\eps(\bar q):=[t_1-\zeta_-(\eps),t_1+\zeta_+(\eps)]$ and $\eps \exp\{i\t(t)\}$ parametrizes an arc of the circle $\pa B_\eps(0)$, chosen in such a way that $u_\eps \in \mathcal{K}$. To fix our minds, we suppose 
\[
\t(t)= \frac{\t_+(t-t_1+\zeta_-(\eps)) - \t_-(t-t_1-\zeta_+(\eps))      }{\zeta_+(\eps)+\zeta_-(\eps)} \qquad \forall t \in T_\eps(\bar q).
\] 
Let us note that $\min_t |q_\eps(t)|=\eps$. 

We want to estimate the difference $\mathcal{M}_h(\bar q) - \mathcal{M}_h(q_\eps)$; this can be done through a direct computation:
\begin{multline}\label{eq302}
\mathcal{M}_h(\bar q) - \mathcal{M}_h(q_\eps)\\= \frac{1}{2} \left(\int_{[c,d] \setminus T_\eps(\bar q)} |\dot{\bar q}|^2 + \int_{T_\eps(\bar q)} |\dot{\bar q}|^2\right) \left( \int_{[c,d] \setminus T_\eps(\bar q)}  \left(V(\bar q)+h\right) + \int_{T_\eps(\bar q)} \left(V(\bar q)+h\right)\right) -\\- \frac{1}{2} \left(\int_{[c,d] \setminus T_\eps(\bar q)} |\dot{\bar q}|^2 + \int_{T_\eps(\bar q)} |\dot{q}_\eps|^2\right)\left( \int_{[c,d]\setminus T_\eps(\bar q)} \left(V(\bar q)+h\right) + \int_{T_\eps(\bar q)} \left(V(q_\eps)+h\right)\right) \\
= \frac{1}{2} \int_{[c,d] \setminus T_\eps(\bar q)} |\dot{\bar q}|^2 \left( \int_{T_\eps(\bar q)} V(\bar q)-V(q_\eps)  \right) +\\
+\frac{1}{2} \int_{[c,d]\setminus T_\eps(\bar q)}  \left(V(\bar q)+h\right) \left(\int_{T_\eps(\bar q)} |\dot{\bar q}|^2- |\dot{q}_\eps|^2 \right) +\\
+ \frac{1}{2} \int_{T_\eps(\bar q)}|\dot{\bar q}|^2\int_{T_\eps(\bar q)} \left(V(\bar q)+h\right)- \frac{1}{2} \int_{T_\eps(\bar q)}|\dot{q}_\eps|^2\int_{T_\eps(\bar q)} \left(V(q_\eps)+h\right) . 
\end{multline}
For every $t \in T_\eps(\bar q)$ we have
\begin{align*}
|V(\bar q(t)) & - V(q_\eps(t))|  \leq \left|\frac{m_k}{\a}\left(\frac{1}{|\bar u(t)|^\a}- \frac{1}{\eps^\a}\right)\right|+ |V_0(|\bar u(t)|, \bar z(t))-V_0(\eps,\bar z(t))| \\
& \leq C (|t-t_1|^{-\frac{2\a}{\a+2}}+\eps^{-\a})+ C| |\bar u(t)|-\eps| \\
& \leq C (|t-t_1|^{-\frac{2\a}{\a+2}}+\eps^{-\a})+ C\eps.,
\end{align*}
where we used the estimates \eqref{eq300} and the boundedness of $\nabla V_0$ in $\overline{\Gamma}_1$. Taking into account equation \eqref{eq301}, for every $\eps \geq 0$ small enough we have
\begin{equation}\label{eq303}
\begin{split}
\int_{T_\eps(\bar q)} \left|V(\bar q(t))  - V(q_\eps(t)) \right| \,dt \leq \int_{T_\eps(\bar q)} (C (|t-t_1|^{-\frac{2\a}{\a+2}}+\eps^{-\a})+ C\eps)\,dt \\
= C \left(\zeta_+(\eps)^{\frac{2-\a}{\a+2}} + \zeta_-(\eps)^{\frac{2-\a}{\a+2}}\right) + \left(\eps^{-\a}+C \eps\right)\left(\zeta_+(\eps)+\zeta_-(\eps) \right) \leq  C \eps^{\frac{2-\a}{2}}.
\end{split}
\end{equation}
Also, for $\eps \geq 0$ sufficiently small
\beq\label{eq304}
\begin{split}
&\int_{T_\eps(\bar q)} |\dot{\bar u}(t)|^2\,dt \leq C \int_{T_\eps(\bar q)} |t-t_1|^{-\frac{2\a}{\a+2}}\,dt \simeq C \eps^{\frac{2-\a}{2}}, \\
&\int_{T_\eps(\bar q)} |\dot{u}_\eps(t)|^2 \,dt = C\int_{T_\eps(\bar q)} \left(\frac{\eps \hat{\t}_1}{\zeta_+(\eps)+\zeta_-(\eps)}\right)^2\,dt \leq C \eps^{\frac{2-\a}{2}}.
\end{split}
\eeq
We can come back to equation \eqref{eq302}: collecting \eqref{eq303} and \eqref{eq304}, for every $\eps \geq 0$ sufficiently small we obtain
\begin{align*}
|\mathcal{M}_h(\bar q)-\mathcal{M}_h(q_\eps)| & \leq C \eps^{\frac{2-\a}{2}} +\frac{1}{2}\int_{T_\eps(\bar q)}|\dot{\bar q}|^2 \int_{T_\eps(\bar q)} \left|V(\bar q)  - V(q_\eps) \right|\\
& \quad + \frac{1}{2} \int_{T_\eps(\bar q)} \left(V(q_\eps)+h\right) \left( \int_{T_\eps(\bar q)}|\dot{q}_\eps|^2+|\dot{\bar q}|^2 \right) \\
& \le C \eps^{\frac{2-\a}{2}} + C\left(  \eps^{\frac{2-\a}{2}} + \eps^{\frac{2+\a}{2}}\right)\eps^{\frac{2-\a}{2}} + C\left( \eps^{-\a}+1) \right)\eps^{\frac{2+\a}{2}} \eps^{\frac{2-\a}{2}}\\
& \le C \eps^{\frac{2-\a}{2}}
\end{align*}
In particular 
\[
d(\eps) \leq \mathcal{M}_h(q_\eps) \leq \mathcal{M}_h(\bar q)+ C \eps^{\frac{2-\a}{2}} = d(0)+ C \eps^{\frac{2-\a}{2}} \qquad \forall \eps \ll 1,
\]
which implies equation \eqref{continuita di d, step 1}. \\
\textbf{Step 2)} \emph{It results 
\begin{equation}\label{cont di d, step 2}
 d(0) \le \liminf_{\eps \to 0^+} d(\eps).
\end{equation}}
Let $(\eps_n)$ be a sequence of positive real numbers such that $\eps_n \to 0$ and $d(\eps_n) \to \liminf_{\eps \to 0^+} d (\eps)$ as $n \to \infty$; by definition, there exists $(q_n) \subset \mathcal{K}$, with $q_n = (u_n,z_n)$, such that
\[
\min_{t \in [c,d]} |u_n(t)|=\eps_n \quad \text{and} \quad \mathcal{M}_h(q_n)=d(\eps_n).
\]
Since $(\mathcal{M}_h(q_n) )$ is bounded and $\mathcal{M}_h$ is coercive in $\mathcal{K}$ (see Lemma \ref{lem: proprieta' di M_h}), the sequence $(q_n)$ is bounded in $H^1$, and therefore, up to a subsequence, it is weakly convergent in $H^1$ to a function $\tilde q \in \mathcal{K}$; recalling that the weak $H^1$ convergence implies the uniform one, we deduce that $\tilde u$ has a collision. This fact, and the weak lower semi-continuity of $\mathcal{M}_h$, imply that
\[
d(0) \leq \mathcal{M}_h(\tilde q) \leq \liminf_{n \to \infty} \mathcal{M}_h(q_n) = \liminf_{\eps \to 0^+} d(\eps). \qedhere
\]
\end{proof}

Now, given $0<\eps_1<\eps_2$, let
\[
\mathcal{K}_{\eps_1,\eps_2}:=\left\lbrace q=(u,z) \in \mathcal{K}: \min_{t \in [c,d]} |u(t)| \in [\eps_1,\eps_2]\right\rbrace .
\]
Since the weak $H^1$ convergence implies the uniform one, it is easy to check that it is a weakly closed subset of $\mathcal{K}$. Let
\[
m(\eps_1,\eps_2):= \inf \left\{ \mathcal{M}_{h}(q): q\in \mathcal{K}_{\eps_1, \eps_2} \right\}.
\]
By Lemma \ref{lem: proprieta' di M_h}, we deduce that the value $m(\eps_1,\eps_2)$ is achieved by a function $q_{\eps_1,\eps_2} \in \mathcal{K}_{\eps_1,\eps_2}$. 

Now, let us consider 
\[
\mathcal{Z}_{\eps_1, \eps_2} :=\left\lbrace  q=(u,z) \in \mathcal{K}_{\eps_1,\eps_2}: \mathcal{M}_h(q)=m(\eps_1,\eps_2) \text{ and }
 \min_{t \in [c,d]} |u(t)| <\eps_2\right\rbrace .
\]

The conclusion of the proof of Theorem \ref{thm: main 1} is a consequence of the following statement. 

\begin{proposition}\label{teorema 2.7}
There exists $\bar{\eps}>0$ such that, if $ 0<\eps_1<\eps_2 \leq \bar{\eps}$, then $\mathcal{Z}_{\eps_1, \eps_2} = \emptyset$.
\end{proposition}

We postpone the proof of this proposition to the next subsection. First, we show how to obtain Theorem \ref{thm: main 1} from it. Let us suppose that Proposition \ref{teorema 2.7} holds. If $0<\eps_1<\eps_2<\eps_3 \leq \bar{\eps}$, then
\[
\text{$q_1=(u_1,z_1)$ is a minimizer of $\mathcal{M}_h$ in $\mathcal{K}_{\eps_2,\eps_3}$} \quad \Longrightarrow \quad \min_{t \in [c,d]} |u_1(t)| = \eps_3,
\]
and
\[
\text{$q_2=(u_2,z_2)$ is a minimizer of $\mathcal{M}_{h}$ in $\mathcal{K}_{\eps_1,\eps_2}$} \quad \Longrightarrow \quad \min_{t \in [c,d]} |u_2(t)| = \eps_2.
\]
Hence $d(\eps_3)<d(\eps_2)<d(\eps_1)$. As the function $d(\cdot)$ is continuous in $0$, taking $\eps_1 \to 0^+$ we obtain $d(\eps_3) < d(\eps_2) \le d(0)$: this is a contradiction, since we are assuming that the minimum of $\mathcal{M}_h$ in $\mathcal{K}$ is achieved over collision paths, and completes the proof of Theorem \ref{thm: main 1}.

\subsection{Proof of Proposition \ref{teorema 2.7}: the blow-up technique}\label{sub: core 3d}

Assume by contradiction that the statement is not true. Then there exist two sequences $(\eps_n)$, $(\bar{\eps}_n)$ converging to $0$, and a sequence $(q_n)$, such that $0<\eps_n<\bar \eps_n$, each $q_n=(u_n,z_n)$ belongs to $\mathcal{K}_{\eps_n, \bar \eps_n}$,
\[
\min_{t \in [c,d]} |u_n(t)|=\eps_n \quad \text{and} \quad \mathcal{M}_h(q_n)=m(\eps_n,\bar{\eps}_n)=d(\eps_n).
\]
From now on, we write $\mathcal K_n$ instead of $\mathcal K_{\eps_n}$ to simplify the notation. Thanks to Lemma \ref{lemma 2.6}, $\mathcal{M}_h(q_n) \to d(0)$ for $n \to \infty$; recalling that we are assuming that the minimum of $\mathcal M_h$ in $\mathcal{K}$ is achieved over collision functions, this means that $(q_n)$ is a minimizing sequence in $\mathcal{K}$. Since $\mathcal M_h$ is coercive, $(q_n)$ is bounded and, up to a subsequence, it is weakly convergent to some $\tilde q \in \mathcal{K}$; by weak lower semi-continuity, $\tilde q$ is a minimizer of $\mathcal{M}_h$ in $\mathcal{K}$, and Lemma \ref{lem: unique minimum} implies that $\tilde q=\bar q$.

The functions $q_n$ enjoy some common properties. Recall that, in cylindrical coordinates, we write $q_n=(u_n,z_n)=(\rho_n \exp{i \theta_n},z_n)$. Let us set
\[
T_{\eps_n}(q_n):=\left\{ t \in (c,d): |u_n(t)|=\eps_n\right\} \qquad T_{\partial}(q_n):=\left\{ t \in (c,d): q_n(t) \in \partial \Gamma_1\right\} 
\]

For every $n$, we define the quantity
\begin{equation}\label{def di omega_n}
\o_n^2:= \frac{\int_c^d \left(V(q_n)+h\right)}{\frac{1}{2} \int_c^d |\dot{q}_n|^2}.
\end{equation}

\begin{lemma}\label{lem: conservation energy q_n}
($i$) The sequence $(\o_n^2)$ is bounded above and below by positive constants. Hence there exist a subsequence of $(q_n)$ (still denoted $(q_n)$) and $\Omega>0$ such that $\omega_n \to \Omega$. \\
($ii$) The energy of the function $q_n$ is constant in $[c,d]$:
\[
\frac{1}{2}|\dot{q}_n(t)|^2-\frac{V(q_n(t))}{\o_n^2}=\frac{h}{\o_n^2} \qquad \text{for a.e. $t \in [c,d]$}.
\]
\end{lemma}
\begin{proof}
($i$) We have
\beq\label{20}
\o_n^2=\frac{\mathcal{M}_h(q_n)}{\frac{1}{4} \left(\int_c^d |\dot{q}_n|^2\right)^2}=\frac{d(\eps_n)}{\frac{1}{4}\|\dot{q}_n\|_{L^2([c,d])}^4}.
\eeq
We know that
\[
0<d(0)<d(\eps_n) \quad \text{and} \quad d(\eps_n) \to d(0) \quad \Longleftrightarrow \quad \exists C_1,C_2 > 0: \ C_1 \leq d(\eps_n) \leq C_2 \ \forall n,
\]
so to prove \eqref{20} it is sufficient to show that $\|\dot{q}_n\|_{L^2([c,d])}$ is bounded from below and from above by positive constants. For $n$ so large that $\eps_n<|p_1-c_k|/2$ any function $q_n$ has to travel at least for the common distance $|p_1-c_k|/2$ in order to pass from $p_1$ to the obstacle $\{|u|=\eps_n\}$; then, arguing by contradiction, it is not difficult to check that there exists $C_3>0$ such that $\|\dot{q}_n\|_2 \ge C_3>0$. Moreover, being $(q_n)$ a minimizing sequence of a coercive functional, it is bounded in the $H^1$ norm. \\
($ii$) It is a consequence of the extremality of $q_n$ with respect to time re-parametrizations keeping the ends fixed. These variations are admissible in $\mathcal{K}_n$.
\end{proof}

\begin{lemma}\label{minimum C^1}
($i$) It results $T_{\partial}(q_n)= \emptyset$ for every $n$;\\
($ii$) For every $n$, the function $q_n$ is of class $\mathcal{C}^1((c,d))$.
\end{lemma}
\begin{proof}
($i$) We already know that $q_n \wc \bar q$ in the weak topology of $H^1$, and hence uniformly in $[c,d]$. As a consequence, the desired result follows simply by the fact that, having chosen $c,d$ so that the function $|\bar u|^2$ is strictly convex in $(c,d)$, it results $T_\pa(\bar q)= \emptyset$. \\
($ii$) By point ($i$), if $q_n$ does not hit the boundary $\pa \Gamma_1$. Now, let 
\[
\mathcal{K}_n'= \left\{ x \in H^1\left(\left[\frac{c}{\o_n},\frac{d}{\o_n} \right]\right): x(s)=q(\o_n s) \text{ for some } q \in \mathcal{K}_n\right\}.
\] 
There is a bijective correspondence between $\mathcal K_n$ and $\mathcal K_n'$, given by 
\[
q(t) \in \mathcal K_n \longleftrightarrow x(s)=q(\o_n s) \in \mathcal K_n'.
\]
We claim that, since $q_n$ is a minimizer of $\mathcal{M}_h$ in $\mathcal K_n$, then $x_n(t) := q_n(\o_n t)$ minimizes the (action) functional
\[
J(x):= \int_{c/\o_n}^{d/\o_n} \left( \frac{1}{2}|\dot{x}|^2+V(x)+h\right)
\]
in the set $\mathcal K_n'$. Indeed, for any $x \in \mathcal K_n'$ corresponding to some $q \in \mathcal{K}_n$, we have 
\begin{align*}
J(x) & \geq 2\left(\int_{c/\o_n}^{d/\o_n}  \frac{1}{2} |\dot{x}(s)|^2\,ds\right)^\frac{1}{2}\left(\int_{c/\o_n}^{d/\o_n} \left(V(x(s))+h\right)\,ds\right)^\frac{1}{2} \\
& = 2 \left(\int_c^d \frac{1}{2} |\dot{q}(t)|^2\,dt \int_0^1 \left(V(q(t))+h\right)\,dt\right)^\frac{1}{2} = 2 \sqrt{\mathcal{M}_h(q)} \geq 2 \sqrt{\mathcal{M}_h(q_n)},
\end{align*}
where the last inequality follows by the minimality of $q_n$ in $\mathcal{K}_n$. As a consequence of the conservation of the energy for $q_n$, Lemma \ref{lem: conservation energy q_n}, we have also
\[
 J(x_n) = \sqrt{\mathcal{M}_h(q_n)},
\]
which proves the claim. 

Now, $q_n \in \mathcal{C}^1((c,d))$ if and only if $x_n \in \mathcal{C}^1\left((c/\omega_n,b/\o_n)\right)$. According to Theorem 1.6 of \cite{MaSc}, if we prove that
\[
\max \left\{0, \limsup_{\substack{ \| x_n-x \|_{2} \to 0 \\ x \in \mathcal K_n'}} \ddfrac{\frac{1}{2}\int_{a/\o_n}^{b/\o_n}|\dot{x}_n|^2-\frac{1}{2}\int_{a/\o_n}^{b/\o_n}|\dot{x}|^2  }{\| x_n-x\|_2} \right\} <+\infty,
\]
then $x_n \in H^2\left(c/\o_n,d/\o_n\right)$ and the proof is complete. For any $x \in \mathcal{K}_n'$, we have
\[
\ddfrac{\frac{1}{2}\int_{c/\o_n}^{d/\o_n}|\dot{x}_n|^2-\frac{1}{2}\int_{c/\o_n}^{d/\o_n}|\dot{x}|^2  }{\|x_n-x\|_{2}}= \frac{J(x_n) - J(x) }{\|x_n-x\|_{2} } + \ddfrac{\int_{c/\o_n}^{d/\o_n} V(x)- V(x_n)  }{\| x_n-x\|_{2}} .
\]
The first term on the right hand side is smaller than $0$ because of the minimality of $x_n$; as far as the second term is concerned, we use the fact that the image of any $x \in \mathcal{K}'_n$ is confined in $\overline{\Gamma}_1 \setminus \left\{|u|<\eps_n\right\}$, and here the potential $V$ is smooth, bounded, and has bounded gradient. Therefore, for any $x \in \mathcal K_n'$
\[
\int_{c/\o_n}^{d/\o_n} V(x)- V(x_n)  \leq C \int_{c/\o_n}^{d/\o_n} | x_n - x| \leq \frac{C}{\sqrt{\omega_n}}\sqrt{d-c}\|x_n - x\|_{2} \to 0,
\]
as $n \to \infty$, where we used point ($i$), and the desired result follows.
\end{proof}

\begin{lemma}\label{proprietà di u_n}
For every $n \in \mathbb{N}$, the function $q_n$ has the following properties:
\begin{itemize}
\item[($i$)] If $(c',d')$ is a connected component of $[c,d] \setminus T_{\eps_n}(q_n)$, then $q_n \in \mathcal{C}^2((c',d'))$, and solves
\beq\label{omega_n}
\o_n^2 \ddot{q}_n(t)=\n V(q_n(t)),
\eeq
Furthermore, the $z$ component $z_n$ is of class $\mathcal{C}^2$ in the whole $(c,d)$, and is a solution of
\[
 \o_n^2 \ddot{z}_n(t)=\partial_z V_0(u_n(t),z_n(t));
\]
\item[($ii$)] For every $n$, there exist $t_n^- \leq t_n^+$ such that:
\begin{align*}
&|u_n(t)| > \eps_n \qquad t \in [c,t_n^-) \cup (t_n^+,d]\\
&|u_n(t)|=\eps_n \qquad t \in [t_n^-,t_n^+],
\end{align*}
that is, $T_{\eps_n}(u_n)=[t_n^-,t_n^+]$;
\item[($iii$)] The function $\theta_n|_{(t_n^-,t_n^+)}$ is of class $\mathcal{C}^2$, is strictly monotone, and is a solution of
\begin{equation}\label{equazione per phi_n}
\dot{\theta}_n(t) \ddot{\theta}_n(t) = \frac{1}{\eps_n^2  \o_n^2} \pa_z V\left(\eps_n,z_n(t) \right) \dot{z}_n(t);
\end{equation}
\item[($iv$)] The angular momentum $\mathfrak{C}_n:= \rho_n^2 \dot \t_n$ is constant in $[c,d]$, and $\mathfrak{C}_n \neq 0$;
\item[($v$)] The sequence $(\dot{z}_n)$ is bounded in $L^\infty([c,d])$.
\end{itemize}
\end{lemma}

\begin{proof}
The proof of ($i$) is analogue to the proof of Lemma \ref{smoothness of z}. \\
($ii$) On every interval $(c',d') \subset (c,d) \setminus T_{\eps_n}(q_n)$, the function $q_n$ solves equation \eqref{omega_n}; using the uniform convergence of $(q_n)$ to $\bar q$ in $[c',d']$, a direct computation of the second derivative of $|u_n|^2|$ (see the proof of Lemma \ref{collisions isolated}) shows that the function it is a strictly convex function in $(c',d')$. Therefore, since $|u_n| \ge \eps_n$ by definition, if there exist $t_1<t_2$ such that $|u_n(t_1)| = |u_n(t_2)|=\eps_n$, then $|u_n(t)|=\eps_n$ for every $t \in (t_1,t_2)$.\\
($iii$) For $t \in (t_n^-,t_n^+)$, the energy integral reads
\begin{equation}\label{energia in polari}
\eps_n^2 \dot{\t}_n^2(t)=\frac{2}{\o_n^2}\left(h+ V\left(u_n(t),z_n(t)\right) \right).
\end{equation}
As a consequence, $\t_n \in \mathcal{C}^2((t_n^-,t_n^+))$. Since $q_n([c,d]) \subset \overline{\Gamma}_1$, and $\overline{\Gamma}_1 \subset \{V>-h\}$ (see Definition \ref{def: local interaction}), equation \eqref{energia in polari} implies that $\dot{\t}_n(t) \neq 0$ for every $t \in (t_n^-,t_n^+)$. To obtain equation \eqref{equazione per phi_n}, it is sufficient to differentiate \eqref{energia in polari} with respect to $t$, and recall that by assumption $ V\left(u_n(t),z_n(t)\right)=V\left(\eps_n,z_n(t)\right)$.\\
($iv$) The conservation of the angular momentum follows by the extremality of $q_n$ with respect to variations of the angle $\t_n$ keeping the radial and the $z$ components fixed: for any $\varphi \in \mathcal{C}^\infty_c([c,d])$, let $q_n^\lambda=(\rho_n \exp\{i (\t_n+\lambda\varphi)\},z_n)$; this is a family of functions in $\mathcal{K}_n$, so that
\[
\left. \frac{d}{d\lambda}\mathcal{M}_h(q_n^\lambda) \right|_{\lambda=0} = 0;
\]
Recalling that $V$ does not depend on $\t$, a direct computation shows that, since this relation holds for every $\varphi \in \mathcal{C}^\infty_c([c,d])$, then $\mathfrak{C}_n$ is constant. Now, assume that $\mathfrak{C}_n=0$; then, since $\rho_n \ge \eps_n$, it must be $\dot{\t}(t)=0$ in $(c,d)$, and in particular $q_n$ has a radial reflection against the obstacle $\{|u|=\eps_n\}$; clearly, in such a situation $q_n$ cannot be of class $\mathcal{C}^1$, in contradiction with Lemma \ref{minimum C^1}. \\
($v$) We know that $z_n$ is a classical solution of $\ddot{z}_n = \pa_z V_0(u_n,z_n)$ in $(c,d)$. The boundedness of $\pa_z V_0$ in $\overline{\Gamma}_1$ implies that there exists $C>0$ such that $\|\ddot{z}_n\|_\infty \le C$. To conclude, it is then sufficient to observe that the sequence $(\dot{z}_n(c))$ is bounded, too. Indeed,
\[
|\dot{z}_n(c)| \le \frac{1}{\sqrt{d-c}} \int_c^d |\dot{z}_n(t)|^2\,dt + \|\pa_z V_0\|_\infty  \le \frac{1}{\sqrt{d-c}} \int_c^d |\dot{q}_n(t)|^2\,dt + C \le C,
\]
where the last inequality is a consequence of the fact that, being $(q_n)$ a minimizing sequence of a coercive functional, it is bounded in $H^1$.
\end{proof}

\begin{remark}\label{rem: su (q_n)}
The monotonicity of $\t_n$, proved in point ($iii$) of the previous lemma, implies that the total variation of the angle $\t_n$ is equal to $\t_n(d)-\t_n(c)=\hat \t_1 \in (0,2\pi/(2-\alpha))$, for any $n$. 
Also, recall that the energy of $q_n$ is constant, and equal to $h/\o_n^2$. Point ($i$) of Lemma \ref{minimum C^1} implies that the sequence of the energies is uniformly bounded from below and from above.
\end{remark}

In the following statement, we give crucial estimates on the angular momentum of $q_n$, and on the amplitude of the time interval in which the function $q_n$ stays on the obstacle $\{|u|=\eps_n\}$.

\begin{lemma}\label{prop 2.12}
The estimates
\[
\mathfrak{C}_n= \eps_n^{\frac{2-\a}{2}} \sqrt{\frac{2 m}{\o_n^2 \a}}\left(1+ O\left(\eps_n^\a\right)\right) \quad \text{and} \quad t_n^+ - t_n^-=O\left(\eps_n^{\frac{\a+2}{2}}\right)
\]
hold for $n \to \infty$.
\end{lemma}
\begin{proof}
For every $t \in [t_n^-,t_n^+]$, it holds 
\beq\label{angular mom on obstacle}
\eps_n |\dot{u}_n(t)|  =\eps_n^2\dot{\t}_n(t).
\eeq
On the other hand, from the expression of the energy and the boundedness of $V_0$ and $(\dot{z}_n)$ (point ($v$) of Lemma \ref{proprietà di u_n}), we obtain
\beq \label{*25}
\begin{split}
\eps_n |\dot{u}_n(t)| &= \eps_n \sqrt{ \frac{2}{\o_n^2} \left(\frac{m}{\a\eps_n^\a}+ V_0\left(\eps_n,z_n(t)\right)+h\right) - \dot{z}_n^2(t)} \\
&= \eps_n^{\frac{2-\a}{2}} \sqrt{\frac{2m}{\o_n^2 \a}+ \frac{2\eps_n^\a}{\o_n^2} \left(V_0(\eps_n,z_n(t))+h- \frac{\o_n^2}{2} \dot{z}_n^2(t)\right)} = \eps_n^{\frac{2-\a}{2}} \sqrt{\frac{2 m}{\o_n^2 \a}}\left(1+O(\eps_n^\a)\right).
\end{split}
\eeq
Plugging into equation \eqref{angular mom on obstacle}, we deduce
\[
\dot{\t}_n(t) = \eps_n^{-\frac{2+\a}{2}} \sqrt{\frac{2 m}{\o_n^2 \a}}\left(1+O(\eps_n^\a)\right),
\]
so that the total variation of $\t_n$ on the obstacle is
\[
\t_n(t_n^+)-\t_n(t_n^-) = \eps_n^{-\frac{2+\a}{2}} \sqrt{\frac{2 m}{\o_n^2 \a}}\left(1+O(\eps_n^\a)\right)(t_n^+-t_n^-).
\]
As observed in Remark \ref{rem: su (q_n)}, this variation is bounded, so that $t_n^+ - t_n^-=O\left(\eps_n^{\frac{\a+2}{2}}\right)$.
\end{proof}

We consider a blow-up of the sequence $(q_n)$. For every $n$, let us fix $t_n \in [t_n^-,t_n^+]$. Recalling that $q_n \to \bar q$ uniformly in $[c,d]$, we deduce that the sequence $(t_n)$ tends to $t_1$, which is the unique collision time of $\bar q$ in $(c,d)$. Let us set
\[
c_n:=\eps_n^{-\frac{\a+2}{2} }(c-t_n), \qquad d_n:=\eps_n^{-\frac{\a+2}{2}} (d-t_n).
\]
We also define
\[
s_n^-:= \eps_n^{-\frac{\a+2}{2} }(t_n^- - t_n), \qquad s_n^+:= \eps_n^{-\frac{\a+2}{2} }(t_n^+ - t_n)
\]
We note that $c_n \to -\infty$, $d_n \to +\infty$ as $n \to \infty$, while $(s_n^-)$ and $(s_n^+)$ are two bounded sequences, thanks to Lemma \ref{prop 2.12}. Hence, up to a subsequence they converge to some limits $s^-$ and $s^+$, respectively.


For every $n$, we define $w_n: [c_n,d_n] \to \R^3$ as
\[
w_n(s):=\left( \frac{1}{\eps_n} u_n\left( t_n +\eps_n^{\frac{\a+2}{2}}s\right), z_n\left( t_n +\eps_n^{\frac{\a+2}{2}}s\right) \right).
\]
Let us observe that if $s \in [s_n^-,s_n^+]$, then $t_n +\eps_n^{(\a+2)/2}s \in [t_n^-,t_n^+]$, and if $s \in [c_n,d_n]$, then $t_n +\eps_n^{(\a+2)/2}s \in [c,d]$. In cylindrical coordinates, we write $w_n= (v_n, \zeta_n) = \left(r_n \exp\{i \phi_n\},\zeta_n\right)$, where
\[
r_n(s)=\frac{1}{\eps_n} \rho_n\left(t_n+\eps_n^{ \frac{\a+2}{2}  }s\right), \quad \phi_n(s)=\t_n\left(t_n+\eps_n^{ \frac{\a+2}{2}  }s\right), \quad
\zeta_n(s) = z_n\left(t_n+\eps_n^{ \frac{\a+2}{2}  }s\right).
\]
In light of Lemma \ref{minimum C^1},  any $w_n$ is of class $\mathcal{C}^1$, and
\[
|v_n(s)| = 1 \quad \text{for $s \in [s_n^-,s_n^+]$}, \qquad |v_n(s)| > 1 \quad \text{for $s \in [c_n,s_n^-) \cup (s_n^+,d_n]$}.
\]
In what follows, we focus on the sequence $(v_n)$ of the planar components of $q_n$. The restrictions of $v_n$ on $[c_n,s_n^-)$ and $(s_n^+,d_n]$ are of class $\mathcal{C}^2$, and satisfy
\[
\ddot{v}_n(s) = -\frac{m}{\o_n^2 \alpha |v_n(s)|^{\alpha+2}} v_n(s) + \eps_n^{\alpha+1} \partial_\rho V_0 \left(w_n(s)\right) = -\frac{m}{\o_n^2 \alpha |v_n(s)|^{\alpha+2}} v_n(s) + O(\eps_n^{\alpha+1}),
\]
where we use the boundedness of $\pa_\rho V_0$ in $\overline{\Gamma}_1$. This suggests to consider the quantity
\[
\bar{h}_n(s):= \frac{1}{2}|\dot{v}_n(s)|^2- \frac{m}{\o_n^2 \a \left|v_n(s)\right|^\a},
\]
which is the energy of the function $v_n$ for the potential of the $\a$-Kepler's problem.  This is not a constant function in $[c_n,d_n]$, however it can be easily controlled using the conservation of the energy of $q_n$:
\begin{align*}
\bar{h}_n(s) & = \eps_n^\a \left[ \frac{1}{2}\left| \dot{u}_n\left( t_n+\eps_n^{ \frac{\a+2}{2}  }s\right)\right|^2- \frac{m}{\o_n^2 \a \left|u_n\left(t_n+\eps_n^{ \frac{\a+2}{2}  }s\right)\right|^\a} \right]\\
& = \eps_n^\a \left[ \frac{h}{\o_n^2} + \frac{1}{\o_n^2} V_0\left(w_n(s)\right)- \frac{1}{2} \dot{\zeta}_n^2(s)\right].
\end{align*}
In light of points ($i$) of Lemma \ref{minimum C^1} and ($v$) of Lemma \ref{proprietà  di u_n}, and of the boundedness of $V_0$, we deduce that
\[
\lim_{n \to \infty} \bar{h}_n(s)=0 \quad \text{for every $s \in \R$},
\]
with uniform convergence in any closed interval of $\R$. 

Concerning the angular momentum of $v_n$, we set $\mathfrak{C}_{v_n}(s):= r_n^2(s) \dot \phi_n(s)$. For every $s \in (s^-_n,s_n^+)$, it results
\[
\mathfrak{C}_{v_n}(s) = \eps_n^{\frac{\alpha+2}{2}} \dot \t_n\left( t_n+\eps_n^{ \frac{\a+2}{2}  }s \right) = \eps_n^{\frac{\alpha-2}{2}}
\mathfrak{C}_{n}\left(t_n+\eps_n^{ \frac{\a+2}{2}  }s \right) = \sqrt{\frac{2m}{\o_n^2 \a}} \left(1+O(\eps_n^{\a+1})\right),
\]
where we used the first estimate of Lemma \ref{prop 2.12}. Hence,
\beq\label{convergenza di c_n}
\lim_{n \to \infty} \mathfrak{C}_{v_n}(s) = \sqrt{\frac{2 m}{\Omega^2 \a}},
\eeq
with uniform convergence in $[s^-,s^+]$ (for the reader's convenience, we recall that $\Omega=\lim_n \o_n$). In particular, the sequence $\left(\mathfrak{C}_{v_n}|_{[s^-,s^+]}\right)$ is uniformly bounded.

Now, let us consider the angular component $\phi_n$ of $v_n$. Starting from point ($iii$) of Lemma \ref{proprietà  di u_n}, and recalling the boundedness of $\pa_z V_0$ and of the sequence $(\dot{z}_n)$ (in $L^\infty(c,d)$), we obtain an equation for $\phi_n$ when $s \in (s_n^-,s_n^+)$:
\[
\ddot{\phi}_n(s) = \frac{\eps_n^{\a}}{\o_n^2 \dot{\phi}_n(s)} \pa_z V_0\left(\eps_n,\zeta_n(s) \right) \dot{\zeta}_n(s).
\]
Observe that in $(s_n^-,s_n^+)$ it results $\mathfrak{C}_{v_n}=\dot{\phi}_n$, and consequently
\[
\ddot{\phi}_n(s) = \frac{\eps_n^{\a}}{\o_n^2 (1+O(\eps_n^{\alpha+1})) } \sqrt{\frac{\omega_n^2 \alpha}{2m}} \pa_z V_0\left(\eps_n,\zeta_n(s) \right) \dot{\zeta}_n(s) = O(\eps_n^\alpha).
\]
Thus the restriction $v_n|_{(s_n^-,s_n^+)}$ is of class $\mathcal{C}^2$ and satisfies
\begin{align*}
\ddot{v}_n(s)  &= \ddot{\phi}_n(s) i e^{i\phi_n(s)}- \left(\dot{\phi}_n(s)\right)^2 e^{i\phi_n(s)} = \ddot{\phi}_n(s) i v_n(s)- \mathfrak{C}_{v_n}(s)^2 v_n(s) \\
&= - \mathfrak{C}_{v_n}(s)^2 v_n(s) + i v_n(s) O(\eps_n^{\a}).
\end{align*}

To sum up up,
\beq\label{equazione per v_n}
\ddot{v}_n = \begin{cases}
                -\ddfrac{m}{\o_n^2 \alpha |v_n(s)|^{\alpha+2}} v_n(s) + O(\eps_n^{\alpha+1})  &  \text{in $[c_n,s_n^-) \cup (s_n^+,d_n]$} \\
                \\
                 - \mathfrak{C}_{v_n}(s)^2 v_n(s) + i v_n(s) O(\eps_n^{\a}) & \text{in $(s_n^-,s_n^+)$}.
                 \end{cases}
\eeq
This shows that, although $v_n$ is not necessarily of class $\mathcal{C}^2$ in $s_n^-$ and $s_n^+$, there exist the right and left limits of the second derivative at these points.

\begin{lemma}\label{convergenza per v_n}
There exists a subsequence of $(v_n)$ which converges in $\mathcal{C}^1_{\loc}(\R)$.
\end{lemma}

\begin{proof}
Let $[a,b] \subset \R$, with $a \le 0\le b$. Equation \eqref{equazione per v_n}, together with the uniform bound on $\mathfrak{C}_{v_n}|_{[s^-,s^+]}$ and the fact that $\inf_{[c_n,d_n]} |v_n|=1$, implies that the sequence $(\ddot{v}_n)$ is bounded in $L^\infty(a,b)$. Moreover, from the expression of the energy $\bar h_n$, we deduce
\begin{align*}
|\dot{v}_n(0)|^2 = 2 \bar h_n(0) + \frac{2m}{\o_n^2 \a |v_n(0)|}; 
\end{align*}
since $|v_n(0)|=1$ for every $n$, and both $(\o_n^2)$ and $(\bar h_n(0))$ are bounded, the sequence $(\dot{v}_n(0))$ is bounded, too. To sum up, $(\ddot{v}_n)$ is bounded in $L^\infty(a,b)$, and, up to a subsequence, both $(\dot{v}_n(0))$ and $(v_n(0))$ converge; hence, it is possible to apply the Ascoli-Arzel\`a theorem, to deduce that (up to a subsequence) $v_n|_{[a,b]} \to v|_{[a,b]}$ in $\mathcal{C}^1(a,b)$. A standard diagonal selection gives the desired result.
\end{proof}

We call $\bar v: \R \to \R^2$ the limit of $(v_n)$ in $\mathcal{C}^1_{\loc}(\R)$. We write $\bar v= \bar r \exp\{i \bar \phi\}$. By equation \eqref{equazione per v_n}, we see that the sequence $(\ddot{v}_n)$ uniformly converges in every compact subset of $\R \setminus \{s^-,s^+\}$, so that $\bar v \in \mathcal{C}^2\left(\R \setminus \{s^-,s^+\}\right)$, and
\begin{itemize}
\item $\bar v$ is a classical solution of the $\a$-Kepler's problem
\[
\ddot{v}(s) = -\frac{m}{ \Omega^2 |\bar v(s)|^{\a+2}}v(s) \quad \text{for $s \in (-\infty,s^-) \cup (s^+,+\infty)$}.
\]
\item $\bar v$ has constant energy, equal to $0$ (even in $[s^-,s^+]$),
\item $\bar v$ has constant angular momentum, whose modulus is $ \mathfrak{C}_{\bar v}=\sqrt{2m/(\Omega ^2 \a)}$ (even in $[s^-,s^+]$); indeed, the equation for $\bar v$ and the \eqref{convergenza di c_n} imply that $\mathfrak{C}_{\bar v}$ is constant in the three intervals $(-\infty,s^-)$, $[s^-,s^+]$ and $(s^+,+\infty)$, and $\mathfrak{C}_{\bar v}=\sqrt{2m/(\Omega ^2 \a)}$ in $[s^-,s^+]$. Using the differentiability of $\bar v$ in $s^\pm$, it is not difficult to prove that the previous expression of $\mathfrak{C}_{\bar v}$ holds in the whole $\R$;
\item $|\bar v(s)|=1$ for $s \in [s^-,s^+]$, and $|\bar v(s)|>1$ for $s \in (-\infty,s^-)\cup (s^+,+\infty)$.
\end{itemize}
Let $\bar \phi^-:=\bar \phi(s^-)$, $\bar \phi^+:=\bar \phi(s^+)$. Thanks to the conservation of the angular momentum, the function $s \mapsto \bar \phi(s)$ is strictly monotone; it is not restrictive to assume that it is increasing, and it makes sense to consider
\[
\bar \phi(+\infty)=\lim_{s \to +\infty} \bar \phi(s), \qquad \bar \phi(-\infty)=\lim_{s \to -\infty} \bar \phi(s).
\]
Writing the energy in polar coordinates, we observe that
\[
ds=\frac{d\bar r}{\sqrt{\frac{2m}{\a \Omega^2 \bar r^\a}- \left(\frac{\mathfrak{C}_{\bar v}}{\bar r}\right)^2 }} = \frac{d\bar r}{\mathfrak{C}_{\bar v} \sqrt{\frac{1}{\bar{r}^{\alpha}} - \frac{1}{\bar r^2}    }}.
\]
Hence
\[
\bar \phi(+\infty)-\bar \phi^+  = \int_{s^+}^{+\infty} \frac{d \bar \phi}{ds}\,ds
 = \int_1^{+\infty} \frac{d\bar r}{\bar r^2\sqrt{\frac{1}{\bar r^\a}-\frac{1}{\bar r^2}}} = \int_0^1 \frac{d\xi}{\sqrt{\xi^\a-\xi^2}}.
\]
The same computation holds true for $\bar \phi^- - \bar \phi(-\infty)$. With the change of variable $\xi= \eta^{\frac{2}{2-\a}}$, we obtain
\begin{align*}
\bar \phi(+\infty)- \bar \phi^+ &= \bar \phi^- - \bar \phi(-\infty) = \frac{2}{2-\a} \int_0^1 \frac{ \eta^{ \frac{\a}{2-\a} } }{ \sqrt{ \eta^{\frac{2\a}{2-\a}} - \eta^{\frac{4}{2-\a} }  } }\,d\eta \\
&=\frac{2}{2-\a} \int_0^1 \frac{ d\eta }{ \sqrt{1-\eta^2} } = \frac{\pi}{2-\a}.
\end{align*}

This permits to obtain the following estimate for the total variation of the angle $\bar \phi$:
\beq\label{*40}
\bar \phi(+\infty)-\bar \phi(-\infty) = \frac{2\pi}{2-\a} +\bar \phi^+- \bar \phi^- \geq \frac{2\pi}{2-\a}.
\eeq
On the other hand, let us note that any function $\phi_n$ is strictly monotone; this is an immediate consequence of the monotonicity of $\t_n$ (point ($iii$) of Lemma \ref{proprietà di u_n}). We know that $\phi_n$ uniformly converges to $\bar \phi$ on every closed interval $[a,b]$ of $\R$. Thus, for $n$ sufficiently large,
\[
\phi_n(b)-\phi_n(a) \leq \phi_n(d_n)-\phi_n(c_n) = \hat \t_1.
\]
Passing to the limit for $n \to \infty$, we deduce that
\[
\bar \phi(b)-\bar \phi(a) \leq \hat \t_1 <\frac{2\pi}{2-\alpha}.
\]
Since $a$ and $b$ are arbitrarily chosen, we can take $a \to -\infty$ and $b \to +\infty$, to obtain
\beq\label{**40}
\bar \phi(+\infty)-\bar \phi(-\infty) < \frac{2\pi}{2-\alpha}.
\eeq
Comparing \eqref{*40} and \eqref{**40} we have a contradiction, and the proof of Proposition \ref{teorema 2.7} is complete.

\section{Proof of Theorem \ref{thm: main 3}}

For a fixed $k$, let $\gamma = \bar u([c,d])$ be a connected component of $\overline{\Xi}_k \cap \bar u([a,b])$. We assume that $[c,d]=[0,1]$ to simplify the notation. Let $\hat \theta$ denote the angle between $\bar u(0)$ and $\bar u(1)$ with respect to the pole $c_k$; we explicitly remark that $|\bar u(0)|= |\bar u(1)|$. As explained in the introduction, having assumed that the interaction between $\bar u$ and $\Sigma$ is locally Keplerian, for every $k$ it is possible to introduce polar coordinates $u = c_k+ \rho \exp\{i \theta\} \in \R^2$, in such a way that 
\[
V(u) = \frac{m_k}{\alpha_k |u-c_k|^{\alpha_k}} + V_0(u) 
\]
for $u$ in $\Xi_k$; we recall that $m_k>0$ and $V_0 \in \mathcal{C}^1\left(\overline{\Xi}_k\right)$. Since we proceed with a local argument, we simply write $m$, $\alpha$, $\bar u$, $\Xi$ instead of $m_k$, $\alpha$, $\bar u|_{[0,1]}$ $\Xi_k$, to simplify the notation. Moreover, for the sake of simplicity, we suppose $c_k=0$.

We are assuming that one of the following alternatives hold true:
\begin{itemize}
\item $\alpha \in (1,2)$ and $\bar u\in K_e$ or in $K_o$;
\item $\alpha=1$, $\bar u(0) \neq \bar u(1)$ and $\bar u \in K_e$ or $K_o$;
\item $\alpha=1$, $\bar u(0)= \bar u(1)$ and $\bar u \in K_o$;
\end{itemize}

All the cases can be treated simultaneously; in what follows, we emphasize when a particular assumption plays a role. We assume that $\bar u$ has a collision in $(0,1)$, and we wish to show that this gives a contradiction with our assumptions. 

As already observed, a lot of results of the previous section still holds true in the present situation. In what follows, we give a brief account of those we can preserve, and those we have to neglect. 

All the results of Subsection \ref{sub: basic 3d} hold true, except Lemma \ref{lem: angolo costante} and the subsequent discussion. Let $T_c(\bar u)$ be the set of the collision times of $\bar u$. If $(c,d)$ is a connected component of $[0,1] \setminus T_c(u)$, then $\bar u|_{(c,d)}$ is a classical solution of 
\begin{equation}\label{equazione per u}
\bar \omega^2 \ddot{\bar u}= \nabla V(\bar u), \quad \text{where} \quad \bar \omega^2 = \frac{\int_c^d (V(\bar u)+h)}{\frac{1}{2} \int_c^d \dot{\bar u}^2}.
\end{equation}
Moreover, it holds
\begin{equation}\label{energia per planare} 
\frac{1}{2} |\dot{\bar u}|^2-\frac{V(\bar u)}{\bar \o^2} =\frac{h}{\bar \o^2} \qquad \text{a.e. in $[0,1]$}.
\end{equation}
The set $T_c(\bar u)$ is discrete and finite.

\begin{lemma}\label{lem: no autointersezioni}
One of the following situation occurs:
\begin{itemize}
\item[($i$)] $\bar u$ parametrizes a path without self-intersections at points different from the centres $c_k$,
\item[($ii$)] $\bar u$ parametrizes a path with self-intersections at point different from the centres; in such a situation, $\bar u$ has at least one collision, and at a certain collision-time $t_1$ there is a reflection:
\[
\bar u(t_1+t) = \bar u(t_1-t)  \qquad \text{for every $t$ in a neighbourhood of $t_1$}.
\]
\end{itemize}
\end{lemma}

\begin{remark}
Point ($ii$) do not imply that, if a collision occurs, $\bar u$ is a collision-ejection minimizer.
\end{remark}

\begin{proof}
Assume that we are not in case ($ii$). Then either we are in case ($i$), or $\bar u$ has a self-intersection in a point $p \in \Xi_k \neq \{c_k\}$, that is, $p=u(t_*)=u(t_{**})$; in this case, if $\bar u$ has a collision, then there is not any reflection with respect to the collision-time. Assume by contradiction that we are in this latter situation. Let $(c,d)$ the connected component of $[0,1] \setminus T_c(u)$ containing $t_*$. We know that $\bar u|_{(c,d)}$ is a $\mathcal{C}^2$ solution of \eqref{equazione per u} in $(c,d)$. Since $\Xi$ is compactly contained in the Hill's region $\{V(u) > -h\}$, the energy integral says that $|\dot{\bar u}(t)|>0$ for almost every $t \in [0,1]$. Hence, both $\dot{u}(t_*)$ and $\dot{u}(t_{**})$ are different from $0$. One of the following alternatives has to occur: $\dot{\bar u}(t_*)$ is transversal to $\dot{\bar u}(t_{**})$, or $\dot{\bar u}(t_*)$ is tangential to $\dot{\bar u}(t_{**})$ with same or opposite direction. In the first two cases, let us define $v:[0,1] \to \R^2$ as follows:
\[
v(t)=\begin{cases}
      \bar u(t) & t \in [0,t_*] \cup (t_{**},1] ,\\
      \bar u\left(\frac{t-t_*}{t_{**}-t_*} t_* +\left(1-\frac{t-t_*}{t_{**}-t_*}\right)t_{**}\right) & t \in (t_*,t_{**}].
      \end{cases}
\]
The function $v$ parametrizes a path with $\bar u([0,1])=v([0,1])$, but it travels along the loop connecting $\bar u(t_*)$ and $\bar u(t_{**})$ with the reversed orientation. This operation does not change the parity of the winding number with respect to the pole $c_k$, so that $v$ stays in $K_e$ if $\bar u$ does, and stays in $K_o$ if $\bar u$ does. We point out that $v$ is also a local minimizer of $\mathcal M_h$, since $\mathcal M_h(u)=\mathcal M_h(v)$.  On the other hand, it is immediate to check that, unless $\dot{u}(t_*)=\dot{u}(t_{**})=0$, $v$ is not of class $\mathcal{C}^1$ in $t_*$ and $t_{**}$. So, we have a new minimizer of $\mathcal M_h$ in $K_e$ or $K_o$, which is collision-free in an interval $(a,d)\ni t_*$, and hence there should be a classical solution of \eqref{equazione per u}; but this is not possible, since $v|_{(a,d)} \notin \mathcal{C}^1((a,d))$, a contradiction. 

It remains to consider the possibility that $\bar u$ has a tangential self-intersection, with $\dot{\bar u}(t_*)=-\dot{\bar u}(t_{**})$; this situation can be easily ruled out by the uniqueness theorem for initial value problem, taking into account the reversibility of the first equation in \eqref{equazione per u} with respect to the involution $t \mapsto -t$: indeed, it turns out that $\bar u(t_*+t) = \bar u(t_{**}-t)$, but, since $\Xi \subset \subset \{V(u) >-h\}$, so that $\dot{\bar u}(t) \neq 0$ for almost every $t \in (0,1)$, this is possible only if we are in case ($ii$) of the statement. 
\end{proof}

Also in the planar case we pass from a global analysis of the minimizer $\bar u$ to a local study in a neighbourhood of a collision. This is possible because the collisions are isolated: if $\bar u$ has a collision at time $t_1$, then there exist $c,d \in [0,1]$ such that $c<t_1<d$, and $t_1$ is the unique collision time in $[c,d]$; we can choose $c$ and $d$ in such a way that 
\begin{itemize}
\item the function $I=|\bar u|^2$ is strictly convex in $(c,d)$;
\item $|\bar u(c)|=|\bar u(d)|$; we set $\hat \rho:= |\bar u(c)|$, and choose $c,d$ so that $\bar u(c), \bar u(d) \in \partial \Xi$.
\end{itemize} 
We set $\hat \t_1:= \bar \t(d)-\bar \t(c)$; since equation \eqref{equazione per u} is reversible with respect to the involution $t \mapsto -t$, it is not restrictive to assume that $\hat \t_1 \ge 0$. Moreover, in light of Lemma \ref{lem: no autointersezioni}, we can assume that $\hat \t_1 \in [0,2\pi]$. 

Let
\[
\wh{\mathcal{K}}:= \left\lbrace u \in H^1\left([c,d]\right) \left| \begin{array}{l}
\text{$|u(t)| \neq 0$ and $u(t) \in \overline{\Xi}$ for every $t \in [c,d]$,}\\
\text{$u(c)=\bar p_1$, $u(d)= \bar p_2$, and the function} \\
\begin{cases} \bar u(t) & t \in [0,c) \cup (d,1] \\ u(t) & t \in [c,d] \end{cases} \text{ belongs to $K_l$,} \\
\end{array}  \right. \right\},
\]
and let $\mathcal{K}$ be its closure with respect to the weak topology of $H^1$. Let
\[
\mathcal{K}_{\eps}:= \left\{ u \in \mathcal{K}:  \min_{t \in [c,d]} |u(t)|=\eps \right\},
\]
and let
\[
d(\eps):= \inf \left\lbrace \mathcal{M}_{h}(u) : u \in \mathcal{K}_\eps \right\rbrace .
\]
Also, for any $0<\eps_1<\eps_2$, let
\[
\mathcal{K}_{\eps_1,\eps_2}:=\left\lbrace u \in \mathcal{K}: \min_{t \in [c,d]} |u(t)| \in [\eps_1,\eps_2]\right\rbrace,
\]
and let
\[
m(\eps_1,\eps_2):= \inf \left\{ \mathcal{M}_{h}(u): u\in \mathcal{K}_{\eps_1, \eps_2} \right\}.
\]
It is not difficult to check that the result of the Subsection \ref{sub: obstacle} hold true in the present situation. The value $d(0)$ is the infimum of $\mathcal{M}_h$ on the collision elements of $\mathcal{K}$, and, by assumption, is achieved by $\bar u|_{[c,d]}$; this is the unique minimizer of $\mathcal{M}_h$ in $\mathcal{K}$. For any $\eps>0$ sufficiently small, the value $d(\eps)$ is achieved by a function $u_\eps \in \mathcal{K}_{\eps}$, and for any $0<\eps_1<\eps_2$, with $\eps_2$ is sufficiently small, the value $m(\eps_1,\eps_2)$ is achieved by $u_{\eps_1, \eps_2} \in \mathcal{K}_{\eps_1, \eps_2}$. The function $\eps \mapsto d(\eps)$ is continuous in $0$.

Concerning Proposition \ref{teorema 2.7}, we have to modify its statement in the following way. We recall that we are assuming that $\bar u$ has a collision in $t_1 \in (c,d)$.

\begin{proposition}\label{prop: core prop in 2D}
\begin{itemize}
\item[($i$)] Let us assume that $\alpha_k \in (1,2)$. There exists $\bar{\eps}>0$ such that, if $ 0<\eps_1<\eps_2 \leq \bar{\eps}$, then $\mathcal{Z}_{\eps_1, \eps_2} = \emptyset$. 
\item[($ii$)] Let us assume that $\alpha_k=1$. Then, one of the following alternative occurs:
\begin{itemize}
\item[($a$)] there exists $\bar{\eps}>0$ such that, if $ 0<\eps_1<\eps_2 \leq \bar{\eps}$, then $\mathcal{Z}_{\eps_1, \eps_2} = \emptyset$;
\item[($b$)] $\bar u$ is a collision-ejection minimizer, with a unique collision in $t_1$. This is possible only if $\bar u(c) = \bar u(d)$.
\end{itemize}
\end{itemize}
\end{proposition}

Theorem \ref{thm: main 3} follows from this proposition: indeed, in cases ($i$) or ($ii$)-($a$) it is possible to argue as in the final part of Subsection \ref{sub: obstacle}, obtaining a contradiction with the fact that $\bar u$ has a collision. If we are in case ($ii$)-($b$), we observe that necessarily $\bar p_1=\bar p_2$; by the uniqueness theorem for the initial value problem, and using the reversibility of the equation \eqref{equazione per u} with respect to the time involution $t \mapsto -t$, we deduce that $\bar u$ is a collision-ejection minimizer in the whole time interval $[0,1]$, with a unique collision.

To prove Proposition \ref{prop: core prop in 2D}, we proceed as in Subsection \ref{sub: core 3d}. Assume that there exist two sequences $(\eps_n)$, $(\bar{\eps}_n)$ converging to $0$, and a sequence $(u_n)$, such that $0<\eps_n<\bar \eps_n$, each $u_n$ belongs to $\mathcal{K}_{\eps_n, \bar \eps_n}$,
\[
\min_{t \in [c,d]} |u_n(t)|=\eps_n \quad \text{and} \quad \mathcal{M}_h(u_n)=m(\eps_n,\bar{\eps}_n)=d(\eps_n).
\]
We show that, if $\alpha \in (1,2)$, we reach a contradiction, while if $\alpha=1$, then necessarily $\bar u|_{[c,d]}$ is a collision-ejection minimizer.
Thanks to Lemma \ref{lemma 2.6}, $\mathcal{M}_h(u_n) \to d(0)$ for $n \to \infty$; as we are assuming that the minimum of $\mathcal M_h$ in $\mathcal{K}$ is achieved over collision functions, this means that $(u_n)$ is a minimizing sequence in $\mathcal{K}$. Since $\mathcal M_h$ is coercive, $(u_n)$ is bounded and, up to a subsequence, it is weakly convergent to some $\tilde u \in \mathcal{K}$; by weak lower semi-continuity, $\tilde u$ is a minimizer of $\mathcal{M}_h$ in $\mathcal{K}$, and Lemma \ref{lem: unique minimum} implies that $\tilde u=\bar u$.

Lemmas \ref{lem: conservation energy q_n} and \ref{minimum C^1} hold true in the present case (with obvious changes). As far as Lemma \ref{proprietà di u_n}, we have to neglect point ($iv$) (and the parts concerning the $z$ component, of course). However, we can bound the variation of the angle $\t_n$ both in $[c,d]$ and in $[t_n^-,t_n^+]$.

\begin{lemma}\label{variazione dell'angolo limitata}
The function $u_n$ is free of self-intersections in $[c,d]$. In particular, $\dot{\t}_n(d)-\dot{\t}_n(c) \le 2\pi$ and $\t_n(t_n^+)-\t_n(t_n^-) \le 2\pi$.
\end{lemma}

\begin{proof}
The function $u_n$ has no self-intersections for $t \in [c,t_n^-) \cup (t_n^+,d]$. The proof is the same of that of Proposition \ref{lem: no autointersezioni}. If $u_n$ has a self-intersection on the obstacle $\left\lbrace |u|=\eps_n \right\rbrace $, the monotonicity of $\t_n$ on the obstacle (point ($iii$) of Lemma \ref{proprietà di u_n}) implies that $u_n$ makes a complete wind around it. But then we can consider the function $v$ which parametrizes the same path of $u_n$, but reverses the orientation on the obstacle. One has $\mathcal{M}_h(u_n)=\mathcal M_h(v)$, so that $v$ is a local minimizer of $\mathcal M_h$ with $\min_{t \in [c,d]} |v(t)| = \eps_n$. By minimality, $v$ satisfies the energy integral and cannot approach the obstacle with velocity $0$. Therefore, it should be a minimizer which is not $\mathcal{C}^1$, a contradiction.
\end{proof}

Having proved that the variation of the angle $\t_n$ is uniformly bounded, Lemma \ref{prop 2.12} holds true. As a consequence, it is possible to introduce the blow-up sequence 
\[
v_n(s) = \frac{1}{\eps_n} u_n\left(t_n+\eps^\frac{\alpha+2}{2}s \right)  = r_n(s) e^{i \phi_n(s)},
\]
where 
\begin{equation}\label{intro di v_n}
r_n(s) = \frac{1}{\eps_n} \rho \left(t_n+\eps^\frac{\alpha+2}{2}s \right) \quad \text{and} \quad \phi_n(s) = \t_n \left(t_n+\eps^\frac{\alpha+2}{2}s \right).
\end{equation}
It is possible to show that, up to a subsequence, $v_n \to \bar v$ in $\mathcal{C}^1_{\loc}(\R)$, where $\bar v= \bar r \exp\{ i \bar \phi\}$ is a solution of the Kepler' problem for $s \in (-\infty,s^-) \cap (s^+,+\infty)$ (for some $s^\pm \in \R$), has energy $0$ and constant angular momentum, equal to $\sqrt{2m/(\Omega^2 \alpha)}$. Let $\bar \phi^\pm :=  \bar \phi(s^\pm)$. As in the previous section, it is not restrictive to assume $\bar \phi^+-\bar \phi^-  \ge0$, and to deduce that
\begin{equation}\label{estimate angle from below}
\bar \phi(+\infty)-\bar \phi(-\infty) = \frac{2\pi}{2-\alpha} + \bar \phi^+-\bar \phi^- \ge \frac{2\pi}{2-\alpha},
\end{equation}
On the other hand, we know that the total variation of $\phi_n$ is smaller than $2\pi$ for every $n$: indeed, this is a consequence of Lemma \ref{variazione dell'angolo limitata} and of the definition of $v_n$. So, by the $\mathcal{C}^1_{\loc}(\R)$ convergence $v_n \to v$, we deduce also that
\begin{equation}\label{estimate angle from above}
\bar \phi(+\infty)-\bar \phi(-\infty) \le 2\pi.
\end{equation}
In case $\alpha \in (1,2)$, the estimates \eqref{estimate angle from below} and \eqref{estimate angle from above} give a contradiction, and this completes the proof of point ($i$) of Proposition \ref{prop: core prop in 2D}. Otherwise, we deduce the following. 
\begin{lemma}\label{lem: angolo se alpha = 1}
Let $\alpha=1$. If we are not in case ($ii$)-($a$) of Proposition \ref{prop: core prop in 2D}, then necessarily
\[
|\bar \phi^+-\bar \phi^-|= 0.
\]
\end{lemma}

\subsection{A Levi-Civita regularization in a variational framework}

We employ the well known Levi-Civita transformation (see \cite{LeCi}) in order to regularize the flow in a neighbourhood of the singularity $c_k$. We explicitly remark that, to make a local argument of this kind useful, it is essential to know that the collisions are isolated.

\begin{definition}
\textbf{(Local Levi-Civita transform).} For every complex-valued continuous function $u$, we define the set $\Lambda(u)$ of the continuous function $w$ such that
\[
u(t)=w^2(\tau(t)),
\]
where we re-parametrize the time as
\[
dt=|w(\tau)|^2\,d\tau.
\]
\end{definition}
The symbols $`` \, ' \, "$ and $``\n_w "$ denote the differentiation with respect to $\tau$ and the gradient in the Levi-Civita space, respectively. If a path $u$ does not collide in $0$, then $\Lambda(u)$ consists in two elements $\pm \sqrt{u(t(\tau))}$. We perform the Levi-Civita transform for the sequence $(u_n)$ previously introduced. So, it is convenient to define
\[
S_n:= \int_c^d \frac{dt}{|u_n(t)|}.
\]
\begin{lemma}\label{su S_n}
The sequence $(S_n)$ is bounded above and bounded below by strictly positive constants. Hence, there exist a subsequence (still denoted $(S_n)$) and a positive $\wt{S}>0$ such that
\[
\lim_{n \to \infty} S_n=\wt{S}.
\]
\end{lemma}
\begin{proof}
Assume by contradiction that $(S_n)$ is not bounded above. In the proof of point ($iii$) of Lemma \ref{proprietà  di u_n}, we showed that $\|\dot{u}_n\|_2 \ge C_3$ for every $n$; since $\int_0^1 V(u_n) \ge S_n$, $(\mathcal{M}_h(u_n))$ is unbounded, in contradiction with the fact that $(u_n)$ is a minimizing sequence. 

Now, recalling that in $\Xi$ it results $|u| \le d_k$, it is easy to check that the sequence $(S_n)$ is also bounded below by a positive constant.
\end{proof}

We define the sets $\Lambda(u_n)$ and $\Lambda(\bar u)$ of the continuous functions $w_n$ and $\bar w$ such that
\[
\begin{split}
&u_n(t)=w_n^2(\tau(t)), \quad \text{where} \quad dt=S_n |w_n(\tau)|^2\,d\tau, \\
&\bar u(t)=\bar w^2(\tau(t)), \quad \text{where} \quad dt=\bar{S} |\bar w(\tau)|^2\,d\tau.
\end{split}
\]

\begin{remark}
We point out that the new time $\tau$ depends on $n$ (we keep in mind this dependence, but we do not write it down to simplify the notation). The time parameters are suitably normalized to work in a common time interval: setting $\tau(c)=0$ for every $n$, the right end of the interval of definition of each function $w_n$ is
\[
\int_0^{\tau(d)} d\tau = \frac{1}{S_n}\int_c^d \frac{dt}{|u_n(t)|}=1,
\]
so that $w_n$ is defined over $[0,1]$. 
\end{remark}

For $w_n \in \Lambda(u_n)$, we set $\tau_n^-:=\tau(t_n^-)$ and $\tau_n^+:=\tau(t_n^+)$. We recall that $t_n^- = \inf \{t \in [c,d]: |u_n(t)|=\eps_n\}$, $t_n^+=\sup \{t \in[c,d]: |u_n(t)|=\eps_n \}$. The constraint $B_{\eps_n}(0)$ corresponds, through the Levi-Civita transformation, to the ball $B_{\sqrt{\eps_n}}(0)$. Hence, $w_n$ satisfies
\begin{align*}
& |w_n(\tau)|>\sqrt{\r_n} \qquad \tau \in [0,\tau_n^-) \cup (\tau_n^+,1]\\
& |w_n(\tau)|=\sqrt{\r_n} \qquad \tau \in [\tau_n^-, \tau_n^+].
\end{align*}
In polar coordinates, we write $w_n(\tau)=\kappa_n(\tau) \exp\{i \sigma_n(\tau)\}$, where $\kappa_n:[0,1] \to \R^+$ and $\sigma_n:[0,1] \to \R$.

\medskip

For every $\eps_1,\eps_2>0$, and any $u \in \mathcal{K}_{\eps_1 \eps_2}$, it results $\Lambda(u)=\{ \pm \sqrt{u} \}$; thus, the map
\[
\Lambda_+(u): u \mapsto \Lambda_+(u): =+\sqrt{u}
\]
is a bijective correspondence between the spaces $(\mathcal{K}_{\eps_1 \eps_2},dt)$ and $(\Lambda_+(\mathcal{K}_{\eps_1 \eps_2}),d\tau)$.

From now on, for every $n$ we set $w_n =\sqrt{u_n}$. Since $u_n \to \bar u$ uniformly in $[c,d]$, it is possible to choose $\bar u \in \Lambda(\bar  u)$ such that $w_n \to \bar w$ uniformly in $[0,1]$. The next lemma establishes the relationship between the variational properties of $u_n$ and $w_n$.

\begin{lemma}\label{nuovo Maupertuis}
The function $w_n$ is a minimizer of the functional
\[
\tilde{\mathcal M}(w):= 4 \int_0^1 |w'(\tau)|^2\,d\tau \int_0^1 \left[m+ \left(V_0\left(w^2(\tau)\right)+h\right)|w(\tau)|^2\right]\,d\tau
\] 
in the set $\Lambda_+(\mathcal{K}_{\eps_1,\eps_2})$ at a strictly positive level.
\end{lemma}

\begin{proof}
Since $(\mathcal{K}_{\eps_1, \eps_2},dt)$ and $(\Lambda_+(\mathcal{K}_{\eps_1, \eps_2}),d\tau)$ are in bijective correspondence, it is sufficient to write the factors of $\mathcal M_h$ in terms of $\tau$ and $w_n$:
\begin{align*}
|\dot{u}_n(t)|^2\,dt &=\left|2 w_n(\tau(t)) w_n'(\tau(t)) \frac{d\tau}{dt}(t)\right|^2\,dt  = \frac{4}{S_n}|w_n'(\tau)|^2\,d\tau,\\
\left(V(u_n(t))+h\right)\,dt &=\left( \frac{m}{|w_n(\tau(t))|^2}+ V_0\left(w_n^2(\tau(t))\right)+h\right)\, dt \\
&= S_n\left[m+\left(V_0\left(w_n^2(\tau)\right)+h\right)|w_n(\tau)|^2\right]\,d\tau. \qedhere
\end{align*}
\end{proof}

\begin{remark}
The functional $\tilde{\mathcal M}_h$ is another Maupertuis' functional, with regular potential
\[
\tilde V(w):= \left(V_0\left(w^2(\tau)\right)+h\right)|w(\tau)|^2.
\]
Its free critical points, suitably re-parametrized, are solutions of $\ddot w= \nabla_w \tilde V(w)$ with energy $m$.
\end{remark}

For every $n$, let
\[
\wt{\o}_n^2:=\frac{\int_0^1 \left[m+ \left(V_0(w_n^2)+h\right)|w_n|^2\right]}{\frac{1}{2}\int_0^1 |w_n'|^2}.
\]

\begin{lemma}
The sequence $(\wt{\o}_n^2)$ is bounded above and bounded below by positive constants. Hence, there exist a subsequence (still denoted $(\wt{\o}_n)$) and $\wt{\Omega}>0$ such that
\[
\lim_{n \to \infty} \wt{\o}_n=\wt{\Omega}.
\]
\end{lemma}
\begin{proof}
It is sufficient to use the computations of Lemma \ref{nuovo Maupertuis}, and recall that the sequences $(\omega_n^2)$ and $(S_n^2)$ are bounded from below and from above by positive constants.
\end{proof}

We can prove the counterpart of Lemma \ref{proprietà di u_n} for the sequence $(w_n)$.

\begin{lemma}\label{lem: prop di w_n}
For every $n$, the function $w_n$ has the following properties:
\begin{itemize}
\item[($i$)] it is of class $\mathcal{C}^1\left((0,1)\right)$;
\item[($ii$)] the restrictions $w_n|_{[0,\tau_n^-)}$ and $w_n|_{(\tau_n^+,1]}$ are $\mathcal{C}^2$ solutions of
\[
\wt{\o}_n^2 w_n''(\tau)= \n_{q_n} \left( \left(V_0\left(w_n^2(\tau)\right)+h\right)|w_n(\tau)|^2\right)-2w_n(\tau);
\]
\item[($i$)] the energy of $w_n$ is constant in $[0,1]$:
\[
\frac{1}{2}|q_n'(\tau)|^2-\frac{1}{\wt{\o}_n^2}\left(V_0\left(q_n^2(\tau)\right)+h\right)|w_n(\tau)|^2=\frac{m}{\wt{\o}_n^2} \qquad \forall \tau \in [0,1];
\]
\item[($iv$)] the variation of the angle on the constraint tends to $0$ as $n \to \infty$:
\[
\lim_{n \to \infty} |\sigma_n(\tau_n^+)-\sigma_n(\tau_n^-)|=0;
\]
\item[($v$)] The amplitude of the time interval employed by $w_n$ on the constraint tends to $0$ as $n \to \infty$:
\[
\lim_{n \to \infty} (\tau_n^+ - \tau_n^-) = 0.
\]
\end{itemize}
\end{lemma}
\begin{proof}
The point ($i$) is obvious, the points ($ii$) and ($iii$) are consequence of the variational characterization of $w_n$, Lemma \ref{nuovo Maupertuis}.\\
($iv$) By definition, $w_n=\kappa_n \exp\{i \sigma_n\}  = \sqrt{\rho_n} \exp\{ i \t_n/2\} = \sqrt{u_n}$; recalling that $\t_n$ is strictly monotone (see point ($iii$) of Lemma \ref{proprietà di u_n}), we deduce that the variation of the angle $\sigma_n$ on the constraint is $|\sigma_n(\tau_n^+) - \sigma_n(\tau_n^-)| = \frac12 |\t_n(t_n^+)-\t_n(t_n^-)|/2$. Recalling the definition of the blow up sequence $v_n$, and in particular the \eqref{intro di v_n}, we observe that
\[
|\t_n(t_n^+)-\t_n(t_n^-)| =  |\phi_n(s_n^+)-\phi_n(s_n^-)|;
\]
hence, passing to the limit as $n \to \infty$, we can use Lemma \ref{lem: angolo se alpha = 1} to deduce
\[
\lim_{n \to \infty} |\sigma_n(\tau_n^+) - \sigma_n(\tau_n^-)|= \frac{1}{2} \lim_{n \to \infty} |\t_n(t_n^+)-\t_n(t_n^-)|  = |\bar \phi^+- \bar \phi^-| = 0.
\]
($v$) It is a consequence of the same property for $u_n$, Lemma \ref{prop 2.12}:
\[
\tau_n^+ - \tau_n^- = \int_{\tau_n^-}^{\tau_n^+} d\,\tau = \int_{t_n^-}^{t_n^+} \frac{dt}{S_n|w_n(\tau(t))|^2}= \frac{t_n^+ - t_n^-}{S_n \eps_n} = \frac{O(1)\eps_n^{\frac{3}{2}}}{S_n \eps_n} \simeq \frac{\eps_n^{\frac{1}{2}}}{S_n} \to 0
\]
for $n \to \infty$, where we used the boundedness of the sequence $(S_n)$, Lemma \ref{su S_n}.
\end{proof}

The previous result, together with the uniform convergence of $w_n$, permits to obtain an equation for the limit function $\bar w$. 

\begin{lemma}
The path $\bar w$ is a classical solution of
\beq\label{equazione per q tilda}
\wt{\Omega}^2 \bar w''(\tau) = \n_{w} \left(\left(V_0\left(\bar w^2(\tau)\right)+h\right)|\bar w(\tau)|^2\right)-2\bar w(\tau) \qquad \forall \tau  \in (0,1).
\eeq
\end{lemma}
\begin{proof}
The point ($v$) of the previous lemma implies that the sequences $(\tau_n^-)$ and $(\tau_n^+)$ converge to some $\tau_1 \in (0,1)$, such that $\bar w(\tau_1)=0$. This instant $\tau_1$ corresponds to the unique collision time $t_1 \in (c,d)$ of the function $\bar u$. We know that $w_n$ uniformly converges to $\bar w$ over $[0,1]$, and it is not difficult to see that $w_n \to\bar w$ in the $\mathcal{C}^1$-topology in any compact subset of $[0,\tau_1) \cup (\tau_1,1]$ (one can easily follow the proof of Proposition \ref{convergenza per v_n}). Since every $w_n$ is $\mathcal{C}^1$, the vector $w_n(\tau)$ is tangent to the circle $\{w \in \mathbb{C}: |w|=\sqrt{\eps}_n\}$ in the time interval $[\tau_n^-,\tau_n^+]$. So, using the fact that the variation of the angle $\sigma_n$ on the constraint tends to $0$ (we refer to point ($iv$) of Lemma \ref{lem: prop di w_n}), we deduce that
\[
\lim_{\tau \to \tau_1^-}\bar w'(\tau)=\lim_{\tau \to \tau_1^+}  \bar w'(\tau),
\]
that is, $\bar w$ passes trough the origin without any change of direction. As a consequence $\bar w \in \mathcal{C}^1((0,1))$, and it turns out to be a (weak, and by regularity strong) solution of \eqref{equazione per q tilda}.
\end{proof}

\begin{proof}[Conclusion of the proof of point ($ii$) of Proposition \ref{prop: core prop in 2D}]
We wish to show that if we are not in case ($ii$)-($a$), then $\bar u(t_1+t)= \bar u(t_1-t)$. Let us consider the functions
\[
\bar w_1(\tau)=\bar w(\tau_1+\tau), \qquad \bar{w}_2(\tau)= -\bar{w}(\tau_1-\tau).
\]
They are both solutions of \eqref{equazione per q tilda}: for $\bar w_1$ this is immediate, for $\bar w_2$ it is not difficult to check, observing that $w \mapsto (V_0(w^2)+h)|w|^2$ is an even function, and as a consequence 
$w \mapsto\nabla_w [(V_0(w^2)+h)|w|^2]$ is odd. Thanks to the regularity of \eqref{equazione per q tilda}, the uniqueness theorem for initial value problems implies that
\[
\bar w(\tau_1+\tau)=-\bar w(\tau_1-\tau) ,
\]
that is, recalling the definition of the Levi-Civita transform, $\bar u(t_1+t)= \bar u(t_1-t)$.
\end{proof}

\begin{remark}
We proved that, if the minimum of the restriction of $\mathcal M_h$ over $K_l$ is achieved by a collision function $\bar u$, then $\bar u$ is an ejection-collision minimizer. To do this, we considered the minimizing sequence $(u_n)$, defined by means of the introduction of the obstacle problems, and then we passed to the limit in the Levi-Civita space. Thanks to the regularity of the transformed problem, we obtained an equation satisfied by the limit, and this implied the collision-ejection condition for the function $\bar u$. A natural question is the following: why did we pass to $w_n \in \Lambda(u_n)$ instead of considering directly a function in $\Lambda(\bar{u})$? The answer is that, since $|\bar u(t_1)| = 0$, the set $\Lambda(\bar u)$ has not two connected components, so that it is not so clear to give a variational characterization of an arbitrary function in $\Lambda(\bar{u})$ (and hence to deduce an equation for an element of this set). On the other hand, the fact that we fixed the choice $w_n= \sqrt{u_n}$ and the uniform convergence of $u_n$ to $\bar u$ allows to show that the sequence $(w_n)$ converges to a uniquely determined $\bar w \in \Lambda(\bar u)$.
\end{remark}

%

\section*{Appendix: Variational principles}

For the reader's convenience, we collect here some known results (without proofs) concerning the action and the Maupertuis functionals. We refer to \cite{AmCZ,FeTe,Vethesis, SoTe}, and to the references therein. This part can be skipped by the expert reader.

In what follows, $V: \R^N \setminus \Sigma \to \R$ is a singular potential which is smooth in the configuration space $\R^N \setminus \Sigma$, and such that $V(x) \to +\infty$ as $\dist(x,\Sigma) \to 0$. We consider the following fixed ends problem:
\beq\label{fixed ends problem}
\begin{cases}
\ddot{q}(t) = \n V(q(t)) & t \in (a,b),\\
q\left(a\right)=p_1, \qquad q\left(b\right)=p_2,
\end{cases}
\eeq
where the time interval $[a,b]$ can be prescribed a priori, or not. Let
\[
\wh{H}=\wh{H}_{p_1 p_2}([a,b]):= \left\{q \in H^1([a,b],\R^N \setminus \Sigma): \ q(a)=p_1 \text{ and } q(b)=t_2\right\},
\]
and let $H_{p_1 p_2}([a,b])$ be its closure with respect to the weak topology of $H^1$, that is,
\[
H=H_{p_1 p_2}([a,b]) := \left\{q \in H^1([a,b],\R^3): \ q(a)=p_1 \text{ and } q(b)=t_2\right\}.
\]

\appendix

\section{Prescribing the time-interval: the action functional}

The action functional (also called \emph{Lagrangian action}) is
\[
\mathcal{A}_{[a,b]}: H^1([a,b],\R^3) \to \R \cup \{+\infty\}, \quad
\mathcal{A}_{[a,b]}(q):= \int_a^b \left(\frac{1}{2} |\dot{q}(t)|^2 + V(q(t)) \right)\,dt.
\]
It is differentiable in $\wh{H}([a,b])$ (seen as an affine space on $H_0^1$) and its critical points are solution to a fixed time-interval problem.

\begin{theorem}\label{action principle}
Let $q \in \wh{H}_{p_1 p_2}\left([a,b]\right)$ be a critical point of $\mathcal A_{[a,b]}$, i.e.
\[
d\mathcal A_{[a,b]}\left(q\right)[\f]=0 \qquad  \forall \f \in H_0^1\left([a,b]\right).
\]
Then $q$ is a classical solution of
\beq \label{P_[a,b]}
\begin{cases}
\ddot{q}(t) = \n V(q(t)) \qquad &t \in \left[a,b\right],\\
q\left(a\right)=p_1, \qquad q\left(b\right)=p_2.
\end{cases}
\eeq
\end{theorem}

The converse of Theorem \ref{action principle} is also true: if $q \in \mathcal{C}^2\left((a,b)\right)$ is a collisions-free solution of \eqref{P_[a,b]}, then $q$ is a critical point of $\mathcal A_{[a,b]}$.\\
In the above statement, if we replace the condition $q \in \wh{H}$ with $q \in H$ we do not obtain a classical solution anymore, because it is possible that a critical point $q$ has some collisions. Of course, trying to use minimization arguments this is a major problem, as $\wh{H}$ is not weakly closed (and consequently in general minimizing sequences do not converge in $\wh{H}$). In general one performs the minimization in some weakly closed set of $H$ and then try to show that, under additional assumptions, the minimizer is collision-free. In this direction, the following result says that, although collision minimizers are not solution of the motion equation, their energy is constant.
\begin{proposition}
If $q \in H$ is a local minimizer of $\mathcal{A}_{[a,b]}$, then
\[
\frac{1}{2} |\dot{q}(t)|^2-V(q(t))= const.  \qquad \text{for a.e.} \ t \in [a,b].
\]
\end{proposition}
Note that the energy function $t \mapsto  |\dot{q}(t)|^2/2-V(q(t))/\o^2$ a priori is defined only for non-collision times. The previous statement, which is a consequence of the extremality of $q$ for $\mathcal A_{[a,b]}$ with respect to time re-parameterizations, says that if $q$ is a minimizer of $\mathcal A_{[a,b]}$ with a collision, then the energy function can be extended to a constant function (almost everywhere) in the whole time interval $[a,b]$.

\section{Prescribing the energy: the Maupertuis' Principle}

Recall the definition of the Maupertuis' functional: 
\[
\mathcal M_h\left([a,b];\cdot\right):  H_{p_1 p_2}\left([a,b]\right) \to \R \cup \left\lbrace +\infty\right\rbrace, \quad
\mathcal  M_h\left([a,b];q\right) := \frac{1}{2}\int_a^b |\dot{q}(t)|^2 \,dt \int_a^b \left(V(q(t))+h\right)\,dt.
\]
If $\mathcal M_h([a,b];q)>0$ both its factors are strictly positive and it makes sense to set
\beq\label{omega}
\o^2:=\frac{\int_a^b \left(V(q)+h\right)}{\frac{1}{2} \int_a^b |\dot{q}|^2}.
\eeq
The Maupertuis' functional is differentiable in $\wh{H}$, and its critical points, suitably re-parametrized, are solutions to a fixed energy problem.

\begin{theorem}\label{teorema 4.1}
Let $q \in \wh{H}_{p_1 p_2}\left([a,b]\right)$ be a critical point of $\mathcal M_h$ at a positive level, that is,
\[
d\mathcal M_h\left([a,b];q\right)[\f]=0 \quad  \forall \f \in H_0^1\left([a,b]\right), \quad \text{and} \quad \mathcal M_h\left([a,b];q\right)>0;
\]
let $\o$ be given by \eqref{omega}.  Then $x(t):=q(\o t)$ is a classical solution of
\beq \label{P_h}
\begin{cases}
\ddot{x}(t) = \n V(x(t)) \qquad &t \in \left[\frac{a}{\o},\frac{b}{\o}\right],\\
\frac{1}{2}|\dot{x}(t)|^2-V(x(t))=h \qquad  &t \in \left[\frac{a}{\o},\frac{b}{\o}\right],\\
x\left(\frac{a}{\o}\right)=p_1, \qquad x\left(\frac{b}{\o}\right)=p_2,
\end{cases}
\eeq
while $q$ itself is a classical solution of
\beq \label{P_u}
\begin{cases}
\o^2 \ddot{q}(t) = \n V(q(t)) \qquad &t \in [a,b],\\
\frac{1}{2}|\dot{q}(t)|^2-\frac{V(q(t))}{\o^2}=\frac{h}{\o^2} \qquad &t \in [a,b],\\
q(a)=p_1, \qquad q(b)=p_2.
\end{cases}
\eeq
\end{theorem}

The converse of Theorem \ref{teorema 4.1} is also true: if $x \in \mathcal{C}^2\left((a',b')\right)$ is a collisions-free solution of \eqref{P_h}, setting $\o=1/(b'-a')$ and $q(t):= x(t/\o)$, it is not difficult to check that $q$ is a classical solution of \eqref{P_u} defined in $[a'/(b'-a'),b'/(b'-a')]=:[a,b]$ and hence a critical point of $\mathcal M_h\left([a,b];\cdot\right)$ at a strictly positive level. Also, the identity
\[
\omega^2=\frac{\int_a^b \left(V(q)+h\right)}{\frac{1}{2}\int_a^b |\dot{q}|^2}
\]
is fulfilled. In the above statement the fact that $q \in \wh{H}$ automatically rules out the possibility that $q$ has a collision. Although collision minimizers are not true critical points of the Maupertuis' functional in $H$, one can recover the conservation of the energy, as for minimizers of the Lagrangian action.

\begin{proposition}\label{conservazione dell'energia}
If $q \in H$ is a local minimizer of $\mathcal M_h$ at a strictly positive level, then
\[
\frac{1}{2} |\dot{q}(t)|^2-\frac{V(q(t))}{\o^2}= \frac{h}{\o^2}  \qquad \text{for a.e.} \ t \in [a,b].
\]
\end{proposition}

\paragraph{The Jacobi metric.} Solutions of the fixed energy problem \eqref{P_h} can be obtained, after a suitable re-parametrization, also as non-constant critical points of the functional
\[
\mathcal L_h(q)= \mathcal L_h\left([a,b];q\right):= \int_a^b \sqrt{\left(V(q(t))+h\right)} |\dot{q}(t)|    \,dt,
\]
which is defined and differentiable for those $q \in H_{p_1 p_2}\left([a,b]\right)$ such that $V(q(t)) > -h$ for every $t \in [a,b]$. We set
\[
H_h=H_h^{p_1 p_2}\left([a,b]\right):= \left\lbrace q \in H : V(q(t)) > -h, |\dot{q}(t)|>0 \text{ for every $t \in [a,b]$}\right\rbrace;
 \]
the domain of $\mathcal{L}_h$ is the closure of $H_h^{p_1 p_2}\left([a,b]\right)$ in the weak topology of $H^1$.

\begin{theorem}\label{teorema su L}
Let $q \in H_h^{p_1 p_2}\left([a,b]\right) \cap \wh{H}_{p_1 p_2}\left([a,b]\right)$ be a non-constant critical point of $\mathcal L_h\left([a,b];\cdot\right)$. Then there exist a re-parametrization $x$ of $q$ which is a classical solution of $\eqref{P_h}$ in a certain time-interval $[0,T]$.
\end{theorem}

The functional $\mathcal L_h$ has an important geometric meaning: the value $\mathcal L_h(q)$ is the length of the curve parametrized by $q \in H_h$ with respect to the \emph{Jacobi metric}:
\[
g_{ij}(q):=\left(V(q)+h\right) \d_{ij}, \quad
\text{where $\d_{ij}$ is the Kronecker delta}. 
\]
This metric makes the Hill's region $\{V(q)+h>0\}$ a Riemannian manifold. Since $\mathcal L_h$ is a length, it is invariant under re-parametrization.

\paragraph{Relationship between $\mathcal L_h$ and $\mathcal M_h$.}

It is possible to establish a correspondence between minimizers of $\mathcal M_h$ at positive level and minimizers of $\mathcal L_h$. For every $q \in H_h$
\beq\label{5 Maupertuis}
\mathcal L_h^2(q) = \left( \int_a^b \sqrt{\left(V(q)+h\right)}|\dot{q}|\right)^2 \leq \int_a^b |\dot{q}|^2 \int_a^b \left(V(q)+h\right) = 2 \mathcal M_h(q),
\eeq
with equality if and only if there exists $\l \in \R$ such that 
\[
|\dot{q}(t)|^2=\l \left(V(q(t))+h\right) \qquad \text{for almost every $t \in [a,b]$}.
\]
Starting from this fact it is not difficult to show that minimizers of $\mathcal M_h$ "are" minimizers of $\mathcal L_h$. To be precise:

\begin{proposition}\label{minimo M<->L}
Let $q \in H_h \cap H$ be a non-constant minimizer of $\mathcal M_h$. Then $q$ is a minimizer of $\mathcal L_h$ in $H_h \cap H$.
If $q \in H_h \cap H$ is a non-constant minimizer of $\mathcal L_h$ then, up to a re-parametrization, $q$ is a minimizer of $\mathcal M_h$ on $H_h \cap H$.
\end{proposition}

The previous proposition is crucial because, in contrast with $\mathcal M_h$, \emph{the functional $\mathcal L_h$ is additive}. This permits to prove the following result.

\begin{proposition}\label{localizzazione dei minimi_M}
Let $u \in H_{p_1 p_2}\left([a,b]\right)$ be a minimizer of $\mathcal L_h\left([a,b];\cdot\right)$, let $[c,d] \subset [a,b]$. Then $u|_{[c,d]}$ is a minimizer of $\mathcal L_h\left([c,d];\cdot\right)$ in $H_{u(c) u(d)}\left([c,d]\right)$. Moreover,
if $u$ is a minimizer of $\mathcal M_h\left([a,b];\cdot\right)$ in $H_{p_1 p_2}\left([a,b]\right)$, then, for any  subinterval $[c,d] \subset [a,b]$, the restriction $u|_{[c,d]}$ is a minimizer of $\mathcal M_h\left([c,d];\cdot\right)$ in $H_{u(c) u(d)}([c,d])$.
\end{proposition}


\footnotesize
\begin{thebibliography}{99}

\bibitem{AmCZ} 	
\newblock A. Ambrosetti and V. Coti Zelati,
\newblock ``Periodic Solutions of Singular Lagrangian Systems,"
\newblock Birkh\"{a}user, 1993.

\bibitem{BaFeTe08}
\newblock V. Barutello, D. L. Ferrario and S. Terracini,
\newblock \emph{On the singularities of generalized solutions to $n$-body-type problems},
\newblock Int. Math. Res. Notices IMRN, \textbf{2008}, Art. ID rnn 069, 78 pp

\bibitem{BaTeVe}
\newblock V. Barutello,  S.Terracini and G. Verzini,
\emph{Entire minimal parabolic trajectories: the planar anisotropic
   Kepler problem}, {Arch. Ration. Mech. Anal.}, {\bf 207}, {2013}, {583--609}.

\bibitem{Br}
\newblock H. Brezis,
\newblock ``Analyse Fonctionnelle, Th\'eorie et Applications,"
\newblock Colletion Math\'ematiques Appliqu\'ees por la Ma\^{i}trise, Massons, Paris, 1983.

\bibitem{Ca}
\newblock R. Castelli,
\newblock ``On the Variational Approach to the One and N-Centre Problem with Weak Forces,"
\newblock Ph.D Thesis, University of Milano--Bicocca, 2009.

\bibitem{Chen1}
\newblock K.-C. Chen, 
\newblock \emph{Existence and minimizing properties of retrograde orbits to the three-body problem with various choices of masses}. 
\newblock Ann. of Math. (2) 167 (2008), 325–-348. 

\bibitem{Chen2}
\newblock K.-C. Chen and Y.-C. Lin, 
\newblock \emph{On action-minimizing retrograde and prograde orbits of the three-body problem}. 
\newblock Comm. Math. Phys. 291 (2009), no. 2, 403–-441. 

\bibitem{Chencin} 
\newblock A. Chenciner,
\newblock \emph{Action minimizing solutions of the newtonian n-body problem: from
homology to symmetry}, 
\newblock August 2002. ICM, Peking

\bibitem{ChVe} A. Chenciner, A. Venturelli,
\emph{Minima de l'int\'egrale d'action du probl\`eme newtonien de $4$
   corps de masses \'egales dans ${\bf R}^3$: orbites ``hip-hop'', [Minima of the action integral of the Newtonian problem of four bodies of equal mass in $\bf R^3$: ``hip-hop'' orbits]}, Celestial Mech. Dynam. Astronom. 77 (2000), no. 2, 139-152 

\bibitem{Fe}
\newblock D. L. Ferrario, 
\newblock \emph{Transitive decomposition of symmetry groups for the n-body problem}. 
\newblock Adv. Math. 213 (2007), no. 2, 763–-784. 

\bibitem{FeTe}
\newblock D. L. Ferrario and S. Terracini,
\newblock \emph{On the existence of collisionless equivariant minimizers for the classical $n$-body problem},
\newblock Invent. Math., \textbf{155} (2004), 305--362.

\bibitem{FuGr}
\newblock G. Fusco and G. F. Gronchi,  
\newblock Platonic polyhedra, periodic orbits and chaotic motions in the N-body problem with non-Newtonian forces,
\newblock Preprint 2013.


\bibitem{FuGrNe} 
\newblock G. Fusco, G. F. Gronchi and P. Negrini, 
\newblock Platonic polyhedra, topological constraints and periodic solutions of the classical N-body problem. 
\newblock Invent. Math. 185 (2011), no. 2, 283–-332. 


\bibitem{knauf1992}
{M.~Klein and A.~Knauf}, {\em Classical planar scattering by coulombic
  potentials}, Lecture Notes in Physics Monographs, Springer-Verlag, Berlin,
  1992.

\bibitem{LeCi}
\newblock T. Levi-Civita,
\newblock \emph{Sur la r\'egularisation du probl\`eme des trois corps},
\newblock Acta Math., \textbf{42} (1920), 99--144.

\bibitem{Mar} 
\newblock C. Marchal, 
\newblock How the method of minimization of action avoids singularities. 
\newblock Celestial Mech. Dynam. Astronom., 83 (2002), 325–-353. 

\bibitem{MaSc}
\newblock A. Marino and D. Scolozzi,
\newblock \emph{Geodesics with obstacles},
\newblock Boll. Un. Mat. Ital. B (6), 2 (1983), 1--31.

\bibitem{SeTe}
\newblock{E. Serra and S. Terracini} 
\newblock Noncollision Solutions to some minimization problems with Keplerian-like potentials
\newblock Nonlinear Anal. 22 (1994), no. 1, 45--62


\bibitem{SeTe2}
\newblock{E. Serra and S. Terracini} 
\newblock Collisionless periodic solutions to some three-body problems
\newblock Arch. Rational Mech. Anal. 120 (1992), no. 4, 305--325

\bibitem{Shy}
\newblock M. Shibayama,
\newblock Variational proof of the existence of the super-eight in the four-body problem,
\newblock Preprint 2013.    

\bibitem{SoTe}
\newblock N. Soave and S. Terracini,
\newblock \emph{Symbolic dynamics for the $N$-centre problem at negative energies},
\newblock Dyscrete Contin. Dyn. Syst., \textbf{32} (2012), 3245--3301. 

\bibitem{Ta1}
\newblock K. Tanaka,
\newblock \emph{Noncollision solutions for a second order singular Hamiltonian system with weak force},
\newblock Ann. Inst. H. Poincar\'e Anal. Non Lin\'eaire \textbf{10} (2) (1993), 215--238.

\bibitem{Ta2}
\newblock K. Tanaka,
\newblock \emph{A prescribed-energy problem for a conservative singular Hamiltonian system}, 
\newblock Arch. Rational Mech. Anal. \textbf{128} (2) (1994), 127--164.


\bibitem{TeVe} 
\newblock S. Terracini and A. Venturelli,
\newblock \emph{Symmetric trajectories for the $2N$-body problem with equal masses},
\newblock Arch. Ration. Mech. Anal., \textbf{184} (2007), 465--493.

\bibitem{Vent} 
\newblock Venturelli, A.
\newblock \emph{Une caract\'erisation variationnelle des solutions de Lagrange du probl\`eme plan des trois corps}. 
\newblock C. R. Acad. Sci. Paris S\'er. I Math. 332 (2001), no. 7, 641–-644. 


\bibitem{Vethesis}
\newblock A. Venturelli,
\newblock Application de la Minimisation de l'Action au Probl\`eme de $N$
Corps Dans le Plan e Dans l\'{E}space,
\newblock Ph.D Thesis, University Paris VII, 2002.


\end{thebibliography}
\end{document}